\documentclass{amsart}

\usepackage{amsmath,amssymb,amsthm}
\usepackage{mathrsfs}

\newtheorem{theorem}{Theorem}[section]
\newtheorem*{theorem*}{Theorem}
\newtheorem{proposition}[theorem]{Proposition}

\newtheorem{lemma}[theorem]{Lemma}
\theoremstyle{definition}
\newtheorem{definition}[theorem]{Definition}
\newtheorem{example}[theorem]{Example}

\newtheorem{remark}[theorem]{Remark}

\numberwithin{equation}{section}
\makeatletter
\renewcommand{\p@enumii}{\theenumi-}
\renewcommand{\p@enumiii}{\theenumi-\theenumii-}
\renewcommand{\p@enumiv}{\theenumi-\theenumii-\theenumiii-}
\makeatother

\newcommand{\hits}{\rightharpoonup}
\newcommand{\hitted}{\leftharpoonup}
\newcommand{\id}{\mathsf{id}}
\newcommand{\coo}{\mathsf{coinv}}
\newcommand{\ord}{\mathsf{ord}}
\renewcommand{\dim}{\mathsf{dim}}
\newcommand{\unit}{1}

\newcommand{\kk}{\Bbbk}
\newcommand{\G}{\mathbf{G}}
\newcommand{\Z}{\mathbb{Z}}
\newcommand{\sigmab}{\pmb{\sigma}}
\newcommand{\eb}{\pmb{e}}
\newcommand{\A}{\mathcal{A}}
\newcommand{\B}{\mathcal{B}}
\newcommand{\BB}{\mathscr{B}}

\newcommand{\DD}{\mathscr{D}}
\newcommand{\HH}{\mathcal{H}}
\newcommand{\AU}{\mathcal{A}}

\renewcommand{\AA}{\mathscr{A}}

\newcommand{\sVect}{\mathcal{SV}}
\newcommand{\YD}[1]{{}^{#1}_{#1}\mathcal{YD}}

\newcommand{\AD}[1]{\mathsf{AD}(#1)}
\newcommand{\SD}[1]{\mathsf{SD}(#1)}

\newcommand{\HopfAuto}{\mathsf{Aut}_{\mathrm{Hopf}}}

\newcommand{\GL}{\mathsf{GL}}

\usepackage{hyperref}
\allowdisplaybreaks

\begin{document}

\title{Pointed Hopf superalgebras of dimension up to $10$}

\author[T.~Shibata]{Taiki Shibata}
\address[T.~Shibata]{Department of Applied Mathematics, Okayama University of Science, 1-1 Ridai-cho Kita-ku Okayama-shi, Okayama 700-0005, JAPAN}
\email{shibata@ous.ac.jp}
\author[K.~Shimizu]{Kenichi Shimizu}
\address[K.~Shimizu]{Department of Mathematical Sciences, Shibaura Institute of Technology, 307 Fukasaku, Minuma-ku, Saitama-shi, Saitama 337-8570, JAPAN}
\email{kshimizu@shibaura-it.ac.jp}
\author[R.~Wakao]{Ryota Wakao}
\address[R.~Wakao]{Department of Applied Mathematics, Okayama University of Science, 1-1 Ridai-cho Kita-ku Okayama-shi, Okayama 700-0005, JAPAN}
\email{r23nda8mr@ous.jp}

\date{\today}

\begin{abstract}
By utilizing the technique introduced in our previous work to construct Hopf superalgebras by reversing the Radford-Majid bosonization, we classify non-semisimple pointed Hopf superalgebras of dimension up to $10$ over an algebraically closed field of characteristic zero.
\end{abstract}

\maketitle

\setcounter{tocdepth}{2}
\tableofcontents

\section{Introduction}
The importance of the study of Hopf algebras in braided monoidal categories was recognized through the investigation of quantum groups and their applications. For example, Nichols algebras have been actively investigated for the classification of pointed Hopf algebras \cite{AndEtiGel01,AndNat01,AndSch98,AngGar19,AngGar19b,CaeDasRai00,GarVay10}, a class of Hopf algebras including quantized enveloping algebras. The category $\sVect$ of superspaces over a field $\kk$ of characteristic $\ne 2$ (see Section~\ref{subsec:Hopf-super}) is a simple but significant example of a braided monoidal category. Hopf algebras in $\sVect$ are called {\it Hopf superalgebras}. The enveloping algebra of a Lie superalgebra and its quantization are important examples of Hopf superalgebras in representation theory, low-dimensional topology, mathematical physics, etc.

The classification problem of finite-dimensional Hopf algebras of a given dimension has been actively studied by many researchers since it was proposed by Kaplansky in 1975; see \cite{BeaGar13} for a survey. The classification problem of finite-dimensional Hopf superalgebras over an algebraically closed field of characteristic zero could also be fundamental and crucial. Some families of finite-dimensional Hopf superalgebras have been studied. For instance, Andruskiewitsch, Etingof and Gelaki \cite{AndEtiGel01} classified a class of triangular Hopf algebras by means of finite supergroups. Andruskiewitsch, Angiono and Yamane \cite{AndAngYam11} developed basic results on finite-dimensional pointed Hopf superalgebras. Regarding the classification problem of finite-dimensional Hopf superalgebras of a given dimension, Aissaoui and Makhlouf \cite{AisMak14} classified those of dimension $2$, $3$ and $4$. However, unlike the case of ordinary Hopf algebras, the systematic study of the classification problem seems to have only just begun. The goal of this paper is to give a complete list of finite-dimensional pointed Hopf superalgebras of dimension up to $10$ over an algebraically closed field of characteristic zero as an intermediate step to classify all Hopf superalgebras of those dimensions.

Our central method was introduced in \cite{ShiWak23}. A Hopf superalgebra $\HH$ can be regarded as a Hopf algebra in the braided monoidal category $\YD{\kk\Z_2}$ of Yetter-Drinfeld modules over $\kk\Z_2$, and hence we obtain a Hopf algebra $\widehat{\HH} := \HH\#\kk\Z_2$ from the Hopf superalgebra $\HH$ by the Radford-Majid bosonization \cite{Maj94,Rad85}. We note that $\HH$ and $\widehat{\HH}$ share many properties:
For example, $\HH$ is semisimple if and only if $\widehat{\HH}$ is. $\HH$ is pointed if and only if $\widehat{\HH}$ is. In \cite{ShiWak23}, for a given Hopf algebra $A$, we discussed how Hopf superalgebras $\HH$ such that $\HH_{\bar1}\neq0$ and $\widehat{\HH} \cong A$ are obtained, where $\HH_{\bar1}$ is the odd part of $\HH$. As a consequence, such Hopf superalgebras are parametrized by {\it super-data} for $A$ (see Definition~\ref{def:SD}). A super-datum is a pair $(g, \alpha)$ of group-like elements $g \in \G(A)$ and $\alpha \in \G(A^*)$ such that $g^2 = 1$, $\alpha^2 = \varepsilon$, $\alpha(g) = -1$, $\alpha\hits a \hitted\alpha =gag$ for all $a\in A$ and $g$ is not central in $A$. Furthermore, two Hopf superalgebras arising from two super-data $(g, \alpha)$ and $(g', \alpha')$ are isomorphic if and only if there is a Hopf algebra automorphism $\varphi$ of $A$ such that $\varphi(g) = g'$ and $\alpha' \circ \varphi = \alpha$. Thus the classification of Hopf superalgebras $\HH$ such that $\HH_{\bar1}\neq0$ and $\widehat{\HH} \cong A$ will be completed by the following strategy: First, determine all super-data for $A$. Second, determine Hopf algebra automorphisms of $A$. Finally, give a presentation of the resulting Hopf superalgebras.

If a classification of, say, pointed Hopf algebras of dimension $n$ is known, then the classification of pointed Hopf superalgebras of dimension $n/2$ will be obtained by the above program. From now on, we work over an algebraically closed field $\kk$ of characteristic zero. Let $p$ be an odd prime number. By exploiting results on Hopf algebras of dimensions $2p$ and $2p^2$ \cite{AndNat01,Mas95,Ng05}, we obtain the following theorem:
\begin{theorem*}[$=$ Theorems~\ref{thm:main1} and \ref{thm:main2}]
Let $\HH$ be a Hopf superalgebra over $\kk$.
\begin{enumerate}
\item If $\dim(\HH)=p$, then $\HH$ is purely even, that is, $\HH_{\bar1}=0$.
\item If $\dim(\HH)=p^2$ and $\HH$ is non-semisimple and pointed, then $\HH$ is purely even.
\end{enumerate}
\end{theorem*}

It is easy to see that the exterior superalgebra $\bigwedge\kk$ is the only (up to isomorphism) Hopf superalgebra of dimension $2$ whose odd part is non-trivial. Thanks to the classification of $8$-dimensional Hopf algebras by \c{S}tefan~\cite{Ste99} and that of $16$-dimensional pointed Hopf algebras by Caenepeel, D\u{a}sc\u{a}lescu and Raianu~\cite{CaeDasRai00}, we classify pointed Hopf superalgebras of dimensions $4$ and $8$. Hence we complete the classification of pointed Hopf superalgebras of dimension up to $10$.

\subsection*{Classification lists}
To describe most of the examples in our classification, we introduce the pointed Hopf superalgebra $\AU(\Gamma,\DD)$ constructed from a finite abelian group $\Gamma$ and a compatible datum $\DD \in (\Gamma\times \widehat{\Gamma}\times \{0,1\}\times \{0,1\})^\theta$ as a kind of super-version of pointed Hopf algebras introduced by Andruskiewitsch and Schneider~\cite{AndSch98}; see Section~\ref{subsec:H(D)} for the precise definition. Here, $\widehat{\Gamma}$ is the character group of $\Gamma$ and $\theta$ is a natural number. Using this notation, we display our classification result below, where $\unit$ denotes the trivial character of $\Gamma$.

\subsubsection*{Dimension $2p$}
Let $p$ denote an odd prime, and let $\zeta_p\in\kk$ denote a fixed primitive $p$-th root of unity. Let $\Gamma$ be the group $C_p=\langle g\mid g^p=1 \rangle$ of order $p$, and let $\chi$ be the character defined by $\chi(g)=\zeta_p$. A complete list of pairwise non-isomorphic non-semisimple pointed Hopf superalgebras $\HH$ of dimension $2p$ satisfying $\HH_{\bar1}\neq0$ is given by Table~\ref{tab:2p-dim} (Theorem~\ref{thm:main3}). For explicit relations between $\HH_{2p}^{(i)}$ ($i\in\{1,2,3,4\}$) and $\AU(\Gamma,\DD)$, see Section~\ref{subsec:AN}.
\begingroup
\renewcommand{\arraystretch}{1.5}
\begin{table}[h] \footnotesize
\begin{tabular}{|c||c|c|c|}\hline
\begin{tabular}{c}Hopf superalgebras $\HH$\\ with $\HH_{\bar1}\neq0$\end{tabular} & $\Gamma$ & $\DD=(g_i,\chi_i,\mu_i;\epsilon_i)_{i=1}^\theta$ & The dual $\HH^*$ of $\HH$ \\ \hline
\hline
$\HH_{2p}^{(1)} = \kk C_p\otimes \bigwedge\kk$ & & $\big( 1, \unit, 0; 1 \big)$ & self-dual \\ \cline{1-1}\cline{3-4}
$\HH_{2p}^{(2)}$ & $C_p$ & $( g^{(p+1)/2}, \unit, 0; 1 )$ & $\HH_{2p}^{(3)}$ \\ \cline{1-1}\cline{3-4}
$\HH_{2p}^{(3)}$ & & $( 1, \chi, 0; 1 )$ & $\HH_{2p}^{(2)}$ \\ \cline{1-1}\cline{3-4}
$\HH_{2p}^{(4)}$ & & $( g^{(p+1)/2}, \unit, 1; 1 )$ & \begin{tabular}{c}non-pointed for $p=3,5$\\(Theorem~\ref{thm:non-ss,non-pt})\end{tabular} \\
\hline
\end{tabular}
\caption{Non-semisimple pointed Hopf superalgebras of dimension $2p$}\label{tab:2p-dim}
\end{table}
\endgroup

\subsubsection*{Dimension $4$}
Let $\Gamma$ be the group $C_2=\langle g\mid g^2=1 \rangle$ of order two, and let $\chi$ be the non-trivial character, that is, $\chi(g)=-1$. A complete list of pairwise non-isomorphic non-semisimple pointed Hopf superalgebras $\HH$ of dimension $4$ satisfying $\HH_{\bar1}\neq0$ is given by Table~\ref{tab:4-dim} (Theorem~\ref{prp:4-dim}).
\begingroup
\renewcommand{\arraystretch}{1.5}
\begin{table}[h] \footnotesize
\begin{tabular}{|c||c|c|c|}\hline
\begin{tabular}{c}Hopf superalgebras $\HH$\\ with $\HH_{\bar1}\neq0$\end{tabular} & $\Gamma$ & $\DD=(g_i,\chi_i,\mu_i;\epsilon_i)_{i=1}^\theta$ & The dual $\HH^*$ of $\HH$ \\ \hline
\hline
$\HH_4^{(1)} = \bigwedge\kk^2$ & $\{1\}$ & $( (1, \unit, 0; 1),\;(1, \unit, 0; 1) )$ & self-dual \\ \hline
$\HH_4^{(2)} = \kk C_2\otimes\bigwedge\kk$ & & $( 1, \unit, 0; 1 )$ & self-dual \\ \cline{1-1}\cline{3-4}
$\HH_4^{(3)}$ & $C_2$ & $( g, \unit, 0; 1 )$ & $\HH_4^{(4)}$ \\ \cline{1-1}\cline{3-4}
$\HH_4^{(4)}$ & & $( 1, \chi, 0; 1 )$ & $\HH_4^{(3)}$ \\
\hline
\end{tabular}
\caption{Non-semisimple pointed Hopf superalgebras of dimension $4$}\label{tab:4-dim}
\end{table}
\endgroup

\subsubsection*{Dimension $8$}
Let $\zeta_4\in\kk$ denote a fixed primitive fourth root of unity. If $\Gamma$ is $C_2=\langle g\mid g^2=1\rangle$ (resp.~$C_2\times C_2=\langle g_1,g_2 \mid g_1^2=g_2^2=1, g_1g_2=g_2g_1 \rangle$, $C_4=\langle g\mid g^4=1\rangle$), then we take $\chi$ (resp.~$\chi_1,\chi_2$, $\chi$) so that $\chi(g)=-1$ (resp.~$\chi_i(g_j)=-(-1)^{i+j}$, $\chi(g)=\zeta_4$). A complete list of pairwise non-isomorphic non-semisimple pointed Hopf superalgebras $\HH$ of dimension $8$ satisfying $\HH_{\bar1}\neq0$ is given by Table~\ref{tab:8-dim} (Theorem~\ref{prp:8-dim}).
\begingroup
\renewcommand{\arraystretch}{1.5}
\begin{table}[h] \footnotesize
\begin{tabular}{|c||c|c|c|}\hline
\begin{tabular}{c}Hopf superalgebras $\HH$\\ with $\HH_{\bar1}\neq0$\end{tabular} & $\Gamma$ & $\DD=(g_i,\chi_i,\mu_i;\epsilon_i)_{i=1}^\theta$ & The dual $\HH^*$ of $\HH$ \\ \hline
\hline
$\HH_8^{(1)} = \bigwedge\kk^3$ & $\{1\}$ & $( (1, \unit, 0; 1),\;(1, \unit, 0; 1),\;(1, \unit, 0; 1) )$ & self-dual \\ \hline
$\HH_8^{(2)} = \kk C_2\otimes \bigwedge\kk^2$ & & $( (1, \unit, 0; 1),\; (1, \unit, 0; 1) )$ & self-dual \\ \cline{1-1}\cline{3-4}
$\HH_8^{(3)}$ & & $( (1, \unit, 0; 1),\; (g, \unit, 0; 1) )$ & $\HH_8^{(4)}$ \\ \cline{1-1}\cline{3-4}
$\HH_8^{(4)}$ & $C_2$ & $( (1, \unit, 0; 1),\; (1, \chi, 0; 1) )$ & $\HH_8^{(3)}$ \\ \cline{1-1}\cline{3-4}
$\HH_8^{(5)}$ & & $( (g, \unit, 0; 1),\; (g, \unit, 0; 1) )$ & $\HH_8^{(6)}$ \\ \cline{1-1}\cline{3-4}
$\HH_8^{(6)}$ & & $( (1, \chi, 0; 1),\; (1, \chi, 0; 1) )$ & $\HH_8^{(5)}$ \\ \cline{1-1}\cline{3-4}
$\HH_8^{(7)} = T_4(-1)\otimes \bigwedge\kk$ & & $( (g, \chi, 0; 0),\; (1, \unit, 0; 1) )$ & self-dual \\ \cline{1-1}\cline{3-4}
\hline
$\HH_8^{(8)} = \kk(C_2\times C_2)\otimes \bigwedge\kk$ & & $( 1, \unit, 0; 1 )$ & self-dual \\ \cline{1-1}\cline{3-4}
$\HH_8^{(9)}$ & $C_2^2$ & $( g_1, \unit, 0; 1 )$ & $\HH_8^{(10)}$ \\ \cline{1-1}\cline{3-4}
$\HH_8^{(10)}$ & & $( 1, \chi_1, 0; 1 )$ & $\HH_8^{(9)}$ \\ \cline{1-1}\cline{3-4}
$\HH_8^{(11)}$ & & $( g_1, \chi_2, 0; 1 )$ & self-dual \\ \cline{1-1}\cline{3-4}
\hline
$\HH_8^{(12)} = \kk C_4\otimes\bigwedge\kk$ & & $( 1, \unit, 0; 1 )$ & self-dual \\ \cline{1-1}\cline{3-4}
$\HH_8^{(13)}$ & & $( g, \unit, 0; 1 )$ & $\HH_8^{(15)}$ \\ \cline{1-1}\cline{3-4}
$\HH_8^{(14)}$ & & $( g^2, \unit, 0; 1 )$ & $\HH_8^{(16)}$ \\ \cline{1-1}\cline{3-4}
$\HH_8^{(15)}$ & $C_4$& $( 1, \chi, 0; 1 )$ & $\HH_8^{(13)}$ \\ \cline{1-1}\cline{3-4}
$\HH_8^{(16)}$ & & $( 1, \chi^2, 0; 1 )$ & $\HH_8^{(14)}$ \\ \cline{1-1}\cline{3-4}
$\HH_8^{(17)}$ & & $( g^2, \chi^2, 0; 1 )$ & self-dual \\ \cline{1-1}\cline{3-4}
$\HH_8^{(18)}$ & & $( g, \unit, 1; 1 )$ & \begin{tabular}{c}non-pointed\\(Theorem~\ref{thm:non-ss,non-pt})\end{tabular}\\
\hline
$\HH_8^{(19)}$ & $C_2$ & Section~\ref{subsubsec:A^6} & non-pointed \\
\hline
\end{tabular}
\caption{Non-semisimple pointed Hopf superalgebras of dimension $8$}\label{tab:8-dim}
\end{table}
\endgroup

\subsection*{Corrections to the published version}
The published version omits one non-semisimple pointed Hopf superalgebra of dimension $8$.
The omission comes from an incorrect relation in the presentation of $A^{(6)}$.
The missing class is denoted by $\HH_8^{(19)}$ and is added to Table~\ref{tab:8-dim}.
Its dual is non-pointed and gives a second class in Theorem~\ref{thm:non-ss,non-pt}(2).
The relevant proofs and the other incorrect formulas have been corrected in this version.

\subsection*{Organization of the paper}
This paper is organized as follows.
In Section~\ref{sec:prelim}, we review definitions and properties of Hopf superalgebras.
Fundamental results on the bosonization, which are useful for the classification of Hopf superalgebras, will be recalled from \cite{ShiWak23} in Sections \ref{subsec:boson-super} and \ref{subsec:SD-SF}.
We say that a Hopf superalgebra $\HH$ is a super-form of a Hopf algebra $A$ if $A \cong \widehat{\HH}$.
In Section~\ref{sec:AD}, we discuss super-forms of finite-dimensional pointed Hopf algebras with abelian coradical and, relying on the classification of such Hopf algebras \cite{AngGar19,AngGar19b}, we show that each such super-form is generated by group-like elements and skew-primitive elements.
As an example, we determine all super-forms of a Taft algebra.

In Sections~\ref{sec:class} and \ref{sec:dual}, we work over an algebraically closed field of characteristic zero.
In Section~\ref{sec:class}, after introducing the pointed Hopf superalgebra $\A(\Gamma,\DD)$, we classify super-forms of pointed Hopf algebras of dimension up to $20$.
Non-semisimple pointed Hopf superalgebras of dimensions $p^2$, $2p$, $4$ and $8$ (where $p$ is an odd prime number) are classified in Theorems~\ref{thm:main2}, \ref{thm:main3}, \ref{prp:4-dim} and \ref{prp:8-dim}.
As a result, we obtain the classification of non-semisimple pointed Hopf superalgebras of dimension up to $10$, as explained above.

In Section~\ref{sec:dual}, the final section, we determine the duals of the Hopf superalgebras appearing in Section~\ref{sec:class}.

\subsection*{Acknowledgment}
T.S. and K.S. are supported by JSPS KAKENHI Grant Numbers JP22K13905 and JP20K03520, respectively.

\section{Preliminaries} \label{sec:prelim}
In this section, we work over a field $\kk$ of characteristic not equal to $2$. The unadorned symbol $\otimes$ means the tensor product over $\kk$. We denote by $\Z_2$ the additive group of integers modulo $2$. The class of $n \in \Z$ in $\Z_2$ is written as $\bar{n}$ or, by abuse of notation, as $n$.

The group algebra of $\Z_2$ is denoted by $\kk\Z_2$. When we consider $\kk\Z_2$, we identify $\Z_2$ with the multiplicative group $\{\eb,\sigmab\}$ of order two, where $\eb$ is the identity element and $\sigmab^2=\eb$, for notational convenience.

Given a vector space $X$, we denote by $X^*$ the dual space of $X$. Let $C$ be a coalgebra with comultiplication $\Delta$. We use the Heyneman-Sweedler notation $\Delta(c) = c_{(1)} \otimes c_{(2)}$ to express the comultiplication of $c \in C$. The dual space $C^*$ is an algebra with respect to the multiplication given by $f g(c) = f(c_{(1)}) g(c_{(2)})$ for $f, g \in C^*$ and $c \in C$. We note that $C$ is a $C^*$-bimodule by the actions given by
\[ f\hits c := c_{(1)} f(c_{(2)}),\quad c \hitted f := f(c_{(1)}) c_{(2)} \quad (f \in C^*, c \in C). \]

\subsection{Hopf superalgebras} \label{subsec:Hopf-super}
We denote by $\sVect$ the category of {\it superspaces}. Namely, an object of this category is a vector space $V = V_{\bar{0}} \oplus V_{\bar{1}}$ graded by the group $\Z_2$ and a morphism is a linear map respecting the $\Z_2$-grading. For a homogeneous element $0\neq v\in V_{\bar0}\cup V_{\bar1}$, we denote its degree by $|v|$. We say that $v\in V$ is an {\it even} (resp.~{\it odd}) element if $|v|=0$ (resp.~$|v|=1$). For simplicity, when we write $|v|$, $v$ is always supposed to be homogeneous. We say that $V \in \sVect$ is {\it purely even} if $V_{\bar{1}} = 0$.

The dual space $V^*$ of $V \in \sVect$ is a superspace by letting $(V^*)_{\bar\epsilon} := (V_{\bar\epsilon})^*$ ($\epsilon\in\{0,1\}$). There is a natural tensor product in the category $\sVect$. The category $\sVect$ is in fact a symmetric tensor category with respect to the natural isomorphism
\[ V\otimes W \longrightarrow W\otimes V; \quad v\otimes w \longmapsto (-1)^{|v||w|} w\otimes v \quad (V,W\in \sVect), \]
called the {\it supersymmetry}.

A {\it superalgebra} (resp.~{\it supercoalgebra, Hopf superalgebra}) is an algebra (resp.~coalgebra, Hopf algebra) in the symmetric tensor category $\sVect$. A left supermodule over a superalgebra $\A$ is a superspace $V$ equipped with a morphism $\A \otimes V \to V$ in $\sVect$ satisfying the associativity and the unit axioms. The definition of a right $\A$-supermodule should be clear. Supercomodules over a supercoalgebra are defined analogously.
\begin{definition}\label{def:ss/pt/chev}
Let $\HH$ be a Hopf superalgebra.
\begin{enumerate}
\item\label{def:ss/pt/chev:(1)}
$\HH$ is said to be {\it (co)semisimple} if the category of left $\HH$-super(co)modules is semisimple.
\item\label{def:ss/pt/chev:(2)}
$\HH$ is said to be {\it pointed} if every simple right $\HH$-supercomodule has dimension one.
\item\label{def:ss/pt/chev:(3)}
$\HH$ is said to have the {\it Chevalley property} if the tensor product of any two simple right $\HH$-supercomodules is semisimple.
\end{enumerate}
\end{definition}

The above definition of the Chevalley property is a super-analogue of the Chevalley property of ordinary Hopf algebras considered in \cite{BeaGar13} and \cite{CalDasMasMen04}. We warn that this term has been used with a different meaning in the literature, including \cite{AndEtiGel01}.

We are especially interested in finite-dimensional non-semisimple pointed Hopf superalgebras. Later, we will show that a finite-dimensional semisimple pointed Hopf superalgebra is purely even and is isomorphic to a group algebra of a finite group (see Proposition~\ref{prp:ss-pt}).

Let $\HH$ be a Hopf superalgebra with comultiplication $\Delta_\HH$ and counit $\varepsilon_\HH$. As in the non-super situation, the set
\[ \G(\HH) := \{ g\in \HH_{\bar0} \mid \varepsilon_\HH(g) = 1, \,\, \Delta_{\HH}(g) = g \otimes g \} \]
becomes a group under the multiplication of $\HH$. An element of $\G(\HH)$ is called a {\it group-like} element of $\HH$. For a fixed $g\in \G(\HH)$, an element $z\in\HH$ is said to be {\it $g$-skew primitive} if it satisfies $\Delta_{\HH}(z) = g\otimes z+ z\otimes1$. A $1$-skew primitive element is simply called a {\it primitive element} of $\HH$.

If $\HH$ is a finite-dimensional Hopf superalgebra, then one can make $\HH^*$ into a Hopf superalgebra, called the {\it dual} Hopf superalgebra of $\HH$, in a way similar to that in the ungraded context. Let $\HH$ and $\A$ be Hopf superalgebras, and let $\langle\;,\;\rangle : \HH\times \A\to \kk$ be a bilinear map satisfying $\langle \HH_{\bar\epsilon}, \A_{\bar\nu} \rangle = 0$ for $\epsilon,\nu\in\{0,1\}$ with $\epsilon\neq\nu$. The map $\langle\;,\;\rangle$ is called a {\it Hopf pairing} if it satisfies
\[ \begin{gathered}
\langle xy, a \rangle = \langle x, a_{(1)} \rangle \langle y, a_{(2)} \rangle, \quad \langle x, ab \rangle = \langle x_{(1)}, a \rangle \langle x_{(2)}, b\rangle, \quad \\
\langle x, 1_{\A} \rangle = \varepsilon_{\HH}(x), \quad \langle 1_{\HH}, a \rangle = \varepsilon_{\A}(a)
\end{gathered} \]
for $x,y\in\HH$ and $a,b\in\A$. If $\langle\;,\;\rangle$ is a Hopf pairing and $\A$ is finite-dimensional, then one sees that the map $\HH\to\A^*$; $x\mapsto (a\mapsto\langle x, a \rangle)$ is a Hopf superalgebra map. Note that there exists another definition of a Hopf pairing taking into account the supersymmetry. However, if the base field $\kk$ is algebraically closed, the two conventions give isomorphic dual Hopf superalgebras; see \cite[Remark~3.10]{ShiWak23}.
\begin{example} \label{ex:ext}
Let $V$ be a vector space. The exterior superalgebra $\bigwedge V$ over $V$ has a natural structure of a pointed Hopf superalgebra so that each $z\in V$ is odd and primitive.
The canonical pairing $\langle\;,\;\rangle:V\times V^*\to\kk$ uniquely extends to a non-degenerate Hopf pairing $\langle\;,\;\rangle:\bigwedge V \times\bigwedge V^*\to\kk$ given by
\[ \langle v_1\wedge \cdots \wedge v_n, f_1\wedge \cdots \wedge f_m \rangle = \delta_{n,m}\; \det\big( \langle v_i, f_j \rangle \big)_{i,j} \qquad (n,m\in\mathbb{N}), \]
where $\delta_{n,m}$ is the Kronecker symbol. In particular, $\bigwedge V$ is self-dual if $V$ is finite-dimensional.
\end{example}

\subsection{Bosonization of Hopf superalgebras} \label{subsec:boson-super}
Let $H$ be a Hopf algebra. By Radford~\cite{Rad85} and Majid~\cite{Maj94}, up to isomorphism, there is a one-to-one correspondence between Hopf algebras equipped with a split epimorphism onto $H$ and Hopf algebras in the category $\YD{H}$ of Yetter-Drinfeld modules over $H$. For a Hopf algebra $\B$ in $\YD{H}$, we denote the corresponding Hopf algebra by $\B\#H$, called the {\it bosonization} of $\B$ by $H$.

By definition, an object of $\YD{\kk\Z_2}$ is an object $V \in \sVect$ equipped with a left action of $\kk\Z_2$ such that $\sigmab.V_{\bar{\epsilon}} \subset V_{\bar{\epsilon}}$ for $\epsilon \in \{0,1\}$. A superspace $V$ becomes an object of $\YD{\kk\Z_2}$ by defining the left action of $\kk\Z_2$ by  $\sigmab^i.v:=(-1)^{i|v|}v$ for $i\in\{0,1\}$ and $v\in V$. In this way, we can regard $\sVect$ as a braided monoidal subcategory of $\YD{\kk\Z_2}$.

Let $\HH$ be a Hopf superalgebra with comultiplication $\Delta_\HH$, counit $\varepsilon_\HH$ and antipode $S_\HH$. Since $\sVect\subset \YD{\kk\Z_2}$, we can consider the bosonization
\[ \widehat{\HH} := \HH \# \kk\Z_2 \]
of $\HH$ by $\kk\Z_2$. More precisely, as a vector space $\widehat{\HH}$ is just $\HH\otimes \kk\Z_2$ and $\widehat{\HH}$ becomes an ordinary Hopf algebra whose structure is described as follows:
\begin{itemize}
\item (multiplication) $(h\otimes\sigmab^i)(h'\otimes\sigmab^j) = (-1)^{i|h'|} h h'\otimes \sigmab^{i+j}$.
\item (unit) $1_{\widehat{\HH}} = 1_\HH\otimes \eb$.
\item (comultiplication) $\Delta_{\widehat{\HH}}(h\otimes\sigmab^i)= h_{(1)}\otimes\sigmab^{i+|h_{(2)}|}\otimes h_{(2)}\otimes\sigmab^i$.
\item (counit) $\varepsilon_{\widehat{\HH}}(h\otimes\sigmab^i) = \varepsilon_\HH(h)$.
\item (antipode) $S_{\widehat{\HH}}(h\otimes\sigmab^i) = (-1)^{(i+|h|)|h|}S_\HH(h)\otimes\sigmab^{i+|h|}$.
\end{itemize}
Here, $h,h'\in\HH$ and $i,j\in\{0,1\}$. One easily sees $\G(\widehat{\HH}) \cong \G(\HH)\times\Z_2$ as groups. If $\HH_{\bar1}\neq0$, then $\widehat{\HH}$ is neither commutative nor cocommutative. As an application of this observation, we obtain:
\begin{theorem}[\text{\cite{ShiWak23}}] \label{thm:main1}
Suppose that $\kk$ has characteristic zero.
All Hopf superalgebras of odd prime dimensions are purely even.
\end{theorem}
\begin{proof}
By extending scalars, we may assume that $\kk$ is algebraically closed. Suppose that there is a Hopf superalgebra $\HH$ of dimension $p$ for some odd prime number $p$ such that $\HH_{\bar1} \neq 0$. By the above observation, $\widehat{\HH}$ is a Hopf algebra of dimension $2p$ that is neither commutative nor cocommutative. However, according to the classification of Hopf algebras of dimension $2p$ due to Masuoka~\cite{Mas95} and Ng~\cite{Ng05}, there is no such Hopf algebra. Thus we have a contradiction.
\end{proof}

The categories of left $\HH$-supermodules and left $\widehat{\HH}$-modules are equivalent.
The same holds for comodules.
Hence we obtain the following.
\begin{proposition} \label{prp:pt/ss/chev}
The following holds.
\begin{enumerate}
\item $\HH$ is (co)semisimple if and only if $\widehat{\HH}$ is (co)semisimple.
\item $\HH$ is pointed if and only if $\widehat{\HH}$ is pointed.
\item $\HH$ has the Chevalley property if and only if $\widehat{\HH}$ has the Chevalley property.
\end{enumerate}
\end{proposition}

As another application of the bosonization technique, we prove:
\begin{proposition} \label{prp:ss-pt}
Suppose that $\kk$ has characteristic zero.
A finite-dimensional semisimple pointed Hopf superalgebra is purely even and isomorphic to the group algebra of a finite group.
\end{proposition}
\begin{proof}
Let $\HH$ be such a Hopf superalgebra. Then the bosonization $\widehat{\HH}$ is a semisimple Hopf algebra. The Larson-Radford theorem implies that $\widehat{\HH}$ is cosemisimple as a coalgebra. This implies that $\HH$ is also cosemisimple as a supercoalgebra. By the assumption that $\HH$ is pointed, $\HH$ is spanned by group-like elements. Hence, $\HH$ is a group algebra.
\end{proof}

Suppose that $\HH$ is finite-dimensional. For the bosonization $\widehat{\HH^*}$ of $\HH^*$, one easily sees that the bilinear map
\[ \widehat{\HH^*} \times \widehat{\HH} \longrightarrow \kk; \quad (f\otimes \sigmab^i, h\otimes \sigmab^j)\longmapsto (-1)^{ij} f(h) \]
is a non-degenerate Hopf pairing. As a consequence, we get the following.
\begin{proposition} \label{prp:dual-hat}
As a Hopf algebra, the bosonization of $\HH^*$ is isomorphic to the dual of $\widehat{\HH}$.
\end{proposition}

\subsection{Coinvariant subalgebras} \label{subsec:SD}
We give a summary of the method for classifying finite-dimensional Hopf superalgebras proposed in \cite{ShiWak23}. Roughly speaking, this is done by enumerating all Hopf superalgebras $\HH$ such that $\widehat{\HH} \cong A$ for each finite-dimensional Hopf algebra $A$. As an intermediate step for accomplishing this, we first mention Hopf algebras $\HH$ in $\YD{\kk\Z_2}$ such that $\widehat{\HH} \cong A$.
\begin{definition}
Let $A$ be a finite-dimensional Hopf algebra with counit $\varepsilon_A$. A pair $(g,\alpha) \in \G(A)\times \G(A^*)$ is called an {\it admissible datum} for $A$ if it satisfies
\[ g^2=1, \quad \alpha^2=\varepsilon_A \quad\text{and}\quad \alpha(g)=-1. \]
The set of all admissible data for $A$ is denoted by $\AD{A}$. For $(g,\alpha),(g',\alpha')\in\AD{A}$, we write $(g,\alpha)\sim(g',\alpha')$ if there exists $\varphi\in\HopfAuto(A)$ such that $\varphi(g)=g'$ and $\alpha=\alpha'\circ\varphi$, where $\HopfAuto(A)$ is the group of all Hopf algebra automorphisms of $A$. It is obvious that the relation $\sim$ becomes an equivalence relation on $\AD{A}$.
\end{definition}

For $(g,\alpha)\in\AD{A}$, one easily sees that the map
\begin{equation}\label{eq:pi}
\pi_{(g,\alpha)} : A\longrightarrow \kk\Z_2; \quad a\longmapsto \frac{\varepsilon_A(a)}{2}(\eb+\sigmab) + \frac{\alpha(a)}{2}(\eb-\sigmab)
\end{equation}
is a split epimorphism with section $\kk\Z_2\to A$; $\sigmab^i \mapsto g^i$. Moreover, any split epimorphism $A\to\kk\Z_2$ arises from an element of $\AD{A}$  in this way \cite{ShiWak23}. For the split epimorphism $\pi := \pi_{(g, \alpha)}$, the algebra
\[ A^{\coo(g,\alpha)} := \{ a \in A \mid a_{(1)} \otimes \pi(a_{(2)}) = a \otimes \eb \} \]
of $\pi$-coinvariants of $A$ is given by the following formula:
\begin{equation} \label{eq:B-form}
A^{\coo(g,\alpha)} = \{b\in A \mid b = \alpha\hits b\}.
\end{equation}
By reversing the bosonization \cite{Maj94,Rad85}, we see that $A^{\coo(g,\alpha)}$ becomes a Hopf algebra in $\YD{\kk\Z_2}$ and get the following result:
\begin{proposition}[\text{\cite{ShiWak23}}] \label{prp:one-to-one}
Let $\mathscr{X}$ be the class of all Hopf algebras in $\YD{\kk\Z_2}$ whose bosonization is isomorphic to $A$ as a Hopf algebra. The assignment $(g,\alpha)\mapsto A^{\coo(g,\alpha)}$ gives a bijection from $\AD{A} \mathord{/} \mathord{\sim}$ to $\mathscr{X} \mathord{/} \mathord{\cong}$.
\end{proposition}

\section{Super-data and super-forms of Hopf algebras} \label{sec:AD}
In this section, we also work over a field $\kk$ of characteristic not equal to $2$.
\subsection{Super-data and super-forms}\label{subsec:SD-SF}
We summarize results obtained in \cite{ShiWak23}. Let $A$ be a finite-dimensional Hopf algebra.
\begin{definition}
We say that a Hopf superalgebra $\HH$ is a {\it super-form} of $A$ if $\widehat{\HH}$ is isomorphic to $A$. If $\HH$ is purely even, then $\HH$ is called a {\it trivial super-form} of $A$.
\end{definition}

By the properties of the bosonization, we obtain the following.
\begin{proposition} \label{prp:super-forms}
Let $\HH$ be a super-form of $A$.
\begin{enumerate}
\item\label{prp:super-forms:(1)} The groups $\G(A)$ and $\G(A^*)$ are decomposed into direct products with $\Z_2$.
\item\label{prp:super-forms:(2)} If the super-form $\HH$ is non-trivial, then $A$ is neither commutative nor cocommutative.
\item\label{prp:super-forms:(3)} If $A$ is pointed (resp.~is (co)semisimple, has the Chevalley property), then $\HH$ is pointed (resp.~is (co)semisimple, has the Chevalley property).
\end{enumerate}
\end{proposition}

By Proposition~\ref{prp:one-to-one}, the isomorphism classes of super-forms of $A$ are in bijection with the equivalence classes of $(g,\alpha)\in\AD{A}$ such that $A^{\coo(g,\alpha)}\in\YD{\kk\Z_2}$ belongs to $\sVect$. We have found the following easy criterion for $A^{\coo(g,\alpha)}$ to be a Hopf superalgebra.
\begin{theorem}[\text{\cite{ShiWak23}}]
$\HH:=A^{\coo(g,\alpha)}$ is a super-form of $A$ if and only if
\[ \alpha\hits a \hitted\alpha = gag \quad\text{for all}\quad a\in A. \]
\end{theorem}

If the above equivalent condition is satisfied, then the $\Z_2$-grading of $\HH$ is given by
\[ \HH_{\bar\epsilon}=\{b\in\HH \mid gbg=(-1)^\epsilon b\} \]
for each $\epsilon\in\{0,1\}$. In particular, the super-form $\HH$ of $A$ is non-trivial if and only if $g\notin Z(A)$, where $Z(A)$ is the center of the algebra $A$.

Taking the above into account, we introduce the following definition:
\begin{definition} \label{def:SD}
An admissible datum $(g,\alpha)$ for $A$ is called a {\it super-datum} for $A$ if $g\notin Z(A)$ and $\alpha\hits a \hitted \alpha = gag$ for all $a\in A$. The set of all super-data for $A$ is denoted by $\SD{A}$.
\end{definition}

Then by definition and Proposition~\ref{prp:one-to-one}, the map $(g,\alpha) \mapsto A^{\coo(g,\alpha)}$ gives a one-to-one correspondence between $\SD{A}\mathord{/}\mathord{\sim}$ and the isomorphism classes of Hopf superalgebras $\HH$ such that $\HH_{\bar1}\neq0$ and $\widehat{\HH}\cong A$.

\subsection{Super-forms of pointed Hopf algebras with abelian coradical} \label{subsec:Taft}
In this subsection, assume that $\kk$ is algebraically closed of characteristic zero.
Let $A$ be a finite-dimensional pointed Hopf algebra such that the group $\G(A)$ is abelian. We note that such Hopf algebras were studied extensively and the classification of them was finally completed by Angiono and Garc\'ia in \cite{AngGar19} based on numerous pioneering works on Nichols algebras and arithmetic root systems; see~\cite{AngGar19b} for a survey. According to the classification results, we see that there exist finitely many group-like elements $c_1,\dots,c_\theta\in\G(A)$ and the same number of $c_i$-skew primitive elements $x_i\in A$ ($i\in\{1,\dots,\theta\}$) such that $A$ is generated by $\G(A)\cup\{x_i\}_{i=1}^\theta$ as an algebra. Moreover, for each $i\in\{1,\dots,\theta\}$, there exists a non-trivial group homomorphism $\chi_i:\G(A)\to\kk^\times$ such that
\[ \gamma x_i \gamma^{-1} = \chi_i(\gamma) x_i \]
for all $\gamma\in\G(A)$. We give the following technical remark:
\begin{lemma} \label{prp:alpha(x)=0}
For all $\alpha\in\G(A^*)$ and $i\in\{1,\dots,\theta\}$, we have $\alpha(x_i)=0$.
\end{lemma}
\begin{proof}
Let $i\in\{1,\dots,\theta\}$. Since $\chi_i$ is non-trivial, there exists $\gamma\in \G(A)$ such that $\chi_i(\gamma)\neq1$. We have $\alpha(x_i) = \alpha(\gamma x_i\gamma^{-1}) = \chi_i(\gamma)\alpha(x_i)$, and hence $\alpha(x_i)$ must be zero. This completes the proof.
\end{proof}

By Lemma~\ref{prp:alpha(x)=0}, the set of all super-data $\SD{A}$ for $A$ is explicitly given by
\begin{equation}\label{eq:SD}
\SD{A} = \left\{(g,\alpha)\in\AD{A} \mathrel{}\middle|\mathrel{}
\begin{array}{c} \text{$g$ does not belong to $Z(A)$ and}\\ \text{$gx_ig = \alpha(c_i)x_i$ for all $i=1,\dots,\theta$} \end{array} \right\}.
\end{equation}
For $\alpha\in\G(A^*)$, we put
\[ \G(A)^{\alpha}:=\{\gamma\in\G(A) \mid \alpha(\gamma)=1\} \]
for simplicity. Note that if $\alpha^2=\varepsilon_A$, then $\alpha(\gamma)=\pm1$ for all $\gamma\in\G(A)$.
\begin{theorem} \label{thm:gen}
Let $(g,\alpha)\in\SD{A}$. As an algebra, the coinvariant subalgebra $\HH:=A^{\coo(g,\alpha)}$ of $A$ is generated by $\G(A)^\alpha\cup\{x_i\}_{i=1}^\theta$. The supercoalgebra structure of $\HH$ is described as follows.
\begin{enumerate}
\item\label{prp:gen-prim:(1)} $\G(\HH) = \G(A)^\alpha$.
\item\label{prp:gen-prim:(2)} For each $i\in\{1,\dots,\theta\}$, $x_i$ is an even $c_i$-skew primitive element of $\HH$ if $\alpha(c_i)=1$.
\item\label{prp:gen-prim:(3)} For each $i\in\{1,\dots,\theta\}$, $x_i$ is an odd $c_ig$-skew primitive element of $\HH$ if $\alpha(c_i)=-1$.
\end{enumerate}
\end{theorem}
\begin{proof}
Let $\HH'$ be the subalgebra of $A$ generated by $\G(A)^\alpha\cup\{x_i\}_{i=1}^\theta$. We have $\alpha\hits \gamma = \alpha(\gamma) \gamma$ for all $\gamma \in \G(A)$ and $\alpha \hits x_i = x_i$ for all $i\in\{1,\dots,\theta\}$ by Lemma~\ref{prp:alpha(x)=0}. Thus $\HH' \subset \HH$. Since $\alpha(g) = -1$, we have $\G(A) = \G(A)^{\alpha} \sqcup g \G(A)^{\alpha}$ and, from this, we deduce $A = \HH' + g \HH'$. By considering the dimension, we conclude $\HH = \HH'$.
\end{proof}

\section{Classification of pointed Hopf superalgebras} \label{sec:class}
In this section, the base field $\kk$ is supposed to be an algebraically closed field of characteristic zero. The aim of this section is to classify non-semisimple pointed Hopf superalgebras of dimension up to $10$.
\subsection{The Hopf superalgebra $\AU(\Gamma,\DD)$} \label{subsec:H(D)}
To describe most of the examples in our classification, we introduce a family of pointed Hopf superalgebras $\AU(\Gamma,\DD)$ in the same fashion as Andruskiewitsch and Schneider~\cite{AndSch98}.

Let $\Gamma$ be a finite abelian group with unit $1$, and let $\widehat{\Gamma}$ be the group of group homomorphisms from $\Gamma$ to $\kk^\times$. The trivial character is denoted by $\unit\in\widehat{\Gamma}$. Let $\theta$ be a natural number. We consider a datum
\[ \DD := ( g_i, \chi_i, \mu_i; \epsilon_i )_{i=1}^\theta \quad \in \big(\Gamma\times\widehat{\Gamma}\times\{0,1\}\times\{0,1\}\big)^\theta \]
satisfying the following three conditions.
\begin{enumerate}
\item \label{prp:H(D):(1)} For each $i\in\{1,\dots,\theta\}$, $N_i>1$, and $N_i\in2\Z$ if $\epsilon_i=1$.
\item \label{prp:H(D):(2)} For each $i\in\{1,\dots,\theta\}$, $\mu_i=0$ if 
$\chi_i^{N_i}\neq\unit$.
\item \label{prp:H(D):(3)} For each $i,j\in\{1,\dots,\theta\}$ with $i\neq j$, $\chi_i(g_j)\chi_j(g_i)=1$.
\end{enumerate}
Here, we put $N_i:=\ord((-1)^{\epsilon_i}\chi_i(g_i))$ for each $i\in\{1,\dots,\theta\}$.

We let $\AU(\Gamma,\DD)$ denote the superalgebra generated by $u_g$ with $|u_g|=0$ ($g\in\Gamma$) and $z_i$ with $|z_i|=\epsilon_i$ ($i\in\{1,\dots,\theta\}$) subject to
\[ \begin{gathered} u_1 = 1,\quad u_g u_h = u_{gh},\quad u_g z_i = \chi_i(g) z_i u_g,\\
z_i^{N_i}=\mu_i(1-u_{g_i}^{N_i}),\quad z_iz_j=(-1)^{\epsilon_i\epsilon_j}\chi_j(g_i)z_jz_i, \end{gathered} \]
where $g,h\in\Gamma$ and $i,j\in\{1,\dots,\theta\}$ with $i\neq j$.
\begin{theorem} \label{prp:H(D)}
The superalgebra $\AU(\Gamma, \DD)$ has a unique structure of a Hopf superalgebra whose comultiplication $\Delta$ is given by
\[ \Delta(u_g) = u_g \otimes u_g \quad (g \in \Gamma), \quad \Delta(z_i) = u_{g_i}\otimes z_i + z_i \otimes 1 \quad (i=1,\dots,\theta). \]
The Hopf superalgebra $\AU(\Gamma, \DD)$ is pointed with $\G(\AU(\Gamma,\DD)) \cong \Gamma$.
\end{theorem}

In the following, we identify $g\in\Gamma$ with $u_g \in \AU(\Gamma, \DD)$ and write $\Gamma \subset \AU(\Gamma, \DD)$.

\begin{remark}
For each $i\in\{1,\dots,\theta\}$, we let $\kk^{\chi_i}_{g_i}[\epsilon_i]$ be a one-dimensional left Yetter-Drinfeld supermodule over $\kk\Gamma$ with basis $z_i$ such that
\[ \text{(parity)}\; |z_i|=\epsilon_i,\quad \text{(action)}\; g.z_i = \chi_i(g)z_i\;\;(g\in\Gamma),\quad \text{(coaction)}\;z_i\mapsto g_i\otimes z_i. \]
Then $V:=\kk^{\chi_1}_{g_1}[\epsilon_1]\oplus \cdots \oplus \kk^{\chi_\theta}_{g_\theta}[\epsilon_\theta]$ becomes a left Yetter-Drinfeld supermodule over $\kk\Gamma$. The associated braiding is \[V\otimes V\to V\otimes V;\quad z_i\otimes z_j \mapsto (-1)^{\epsilon_i \epsilon_j} \chi_j(g_i) z_j\otimes z_i.\] By \cite[Section~1.7]{AndAngYam11}, we obtain the {\it Nichols superalgebra} $\BB(V)$ of $V$ and we may consider the bosonization $\BB(V)\#\kk\Gamma$ of $\BB(V)$ by $\kk\Gamma$. One easily sees that the graded Hopf superalgebra of $\AU(\Gamma,\DD)$ associated to the coradical filtration is isomorphic to the Hopf superalgebra $\BB(V)\#\kk\Gamma$.
\end{remark}

\begin{example}
The exterior superalgebra (Example~\ref{ex:ext}) of a vector space of dimension $\theta$ is isomorphic to $\AU(\Gamma,\DD)$ with $\Gamma=\{1\}$ and $\DD=(1,\unit,0;1)_{i=1}^\theta$.
\end{example}
\begin{example}[Super-forms of the Taft algebra]
We fix a natural number $n\geq2$ and a primitive $n$-th root of unity $\omega\in\kk$. We shall write down all non-trivial super-forms of the {\it Taft algebra}
\[ T_{n^2}(\omega) := \kk\langle c,x \mid c^n=1, x^n=0, cx=\omega xc\rangle, \]
where $c$ is group-like and $x$ is $c$-skew primitive. $T_{n^2}(\omega)$ is a non-semisimple pointed Hopf algebra of dimension $n^2$ such that $\G(T_{n^2}(\omega)) \cong \Z_n$. It is easy to see that
\[ \AD{T_{n^2}(\omega)} =\SD{T_{n^2}(\omega)} = \begin{cases} \{(c^{n/2},\alpha)\} & \text{if $n$ is even and $n/2$ is odd,} \\ \varnothing & \text{otherwise.} \end{cases} \]
Here, $\alpha:T_{n^2}(\omega)\to\kk$ is defined by $\alpha(c)=-1$ and $\alpha(x)=0$. In the following, we suppose that $n$ is even and $n/2$ is odd. Let $\HH$ be the super-form of $T_{n^2}(\omega)$ associated to the super-datum $(c^{n/2}, \alpha)$. Then there is an isomorphism
\begin{equation*}
\mathcal{T}_n(\omega^2) := \AU(C_{n/2}, (g^{(n+2)/4},\chi,0;1)) \cong \HH; \quad g\mapsto c^2, \quad z\mapsto x
\end{equation*}
of Hopf superalgebras, where $\chi$ is the character of $C_{n/2} = \langle g \mid g^{n/2} = 1 \rangle$ defined by $\chi(g)=\omega^2$.

By the above discussion, we conclude that $\mathcal{T}_n(\omega^2)$ is the unique super-form of $T_{n^2}(\omega)$ up to isomorphism. By Proposition~\ref{prp:dual-hat} and the self-duality of $T_{n^2}(\omega)$, the Hopf superalgebra $\mathcal{T}_n(\omega^2)$ is self-dual. As a side note, $\mathcal{T}_2(1)$ is isomorphic to the exterior superalgebra $\bigwedge \kk$.
\end{example}

\subsection{Dimensions $2p$ and $p^2$} \label{subsec:AN}
In this section, we discuss super-forms of the Hopf algebras $\AA(\omega,j,\mu)$ introduced by Andruskiewitsch and Natale~\cite[Appendix]{AndNat01}. As an application, we will classify Hopf superalgebras of dimensions $2p$ and $p^2$, where $p$ is an odd prime number.

Let $\ell$ and $q$ be two distinct prime numbers, let $j\in\{1,\ell r \mid r=1,\dots,q-1\}$, and let $\omega$ be a root of unity such that $\omega^{\ell q}=1$ and $\ord(\omega^j) = q$. Let $\mu\in\{0,1\}$ also be a parameter, which is allowed to be non-zero only if $j=1$. Then the algebra
\[ \AA(\omega,j,\mu) := \kk\langle c,x \mid c^{\ell q}=1, x^q=\mu(1-c^q), cx=\omega xc \rangle \]
is a pointed Hopf algebra of dimension $\ell q^2$ such that
\[ \Delta(c) = c \otimes c, \quad \Delta(x) = c^j \otimes x + x \otimes 1. \]
Below we determine the set of admissible data and the set of super-data for this Hopf algebra. As a preparation, we prove:
\begin{lemma} \label{prp:G(A(w,j,mu)^*)}
There exists an algebra map $A:=\AA(\omega,j,\mu)\to\kk$ of order two if and only if either $\ell=2$ and $\mu=0$, or $q=2$. If this condition is satisfied, the algebra map is uniquely determined by $\alpha(c)=-1$ and $\alpha(x)=0$.
\end{lemma}
\begin{proof}
Suppose that $\alpha:A\to\kk$ is an algebra map of order two. By Lemma~\ref{prp:alpha(x)=0}, we have $\alpha(x)=0$. Since $\alpha(c)^2=1$ and $\alpha\ne\varepsilon$, we have $\alpha(c)=-1$. Applying $\alpha$ to the defining relations gives
\[
 (-1)^{\ell q}=1,\qquad \mu(1-(-1)^q)=0.
\]
As $\ell$ and $q$ are distinct primes, these conditions are equivalent to either $\ell=2$ and $\mu=0$, or $q=2$. Conversely, under these conditions the assignments $\alpha(c)=-1$ and $\alpha(x)=0$ preserve all defining relations and give an algebra map of order two.
\end{proof}

When the character in Lemma~\ref{prp:G(A(w,j,mu)^*)} exists, denote it by $\alpha$. Since $\G(\AA(\omega,j,\mu))\cong\Z_{\ell q}$, we have
\[
\AD{\AA(\omega,j,\mu)}=
\begin{cases}
\{(c^q,\alpha)\} & \text{if $\ell=2$ and $\mu=0$,}\\
\{(c^\ell,\alpha)\} & \text{if $q=2$,}\\
\varnothing & \text{otherwise.}
\end{cases}
\]
\begin{proposition} \label{prp:A(w,j,mu)}
The Hopf algebra $A:=\AA(\omega,j,\mu)$ admits a non-trivial super-form if and only if $q=2$. If $q=2$, then $\SD{A}=\{(c^\ell,\alpha)\}$.
\end{proposition}
\begin{proof}
If $\AD{A}=\varnothing$, there is nothing to prove. Suppose first that $\ell=2$ and $\mu=0$, so that $q$ is odd and $\AD{A}=\{(c^q,\alpha)\}$. If $\omega^q=1$, then $c^q$ is central in $A$, and hence $(c^q,\alpha)$ is not a super-datum. If $\omega^q=-1$, then $j\ne1$, because $j=1$ would imply $\ord(\omega)=q$. Thus $j$ is even, and
\[
 c^qxc^q=-x\ne x=\alpha(c^j)x.
\]
It follows that $\SD{A}=\varnothing$.

Suppose next that $q=2$. Then $\ell$ is odd, $j\in\{1,\ell\}$, and $\omega^\ell=-1$. Therefore
\[
 c^\ell xc^\ell=-x=\alpha(c^j)x.
\]
The element $c^\ell$ is not central, so \eqref{eq:SD} gives $\SD{A}=\{(c^\ell,\alpha)\}$.
\end{proof}
\begin{theorem}\label{thm:main2}
Let $p$ be an odd prime number. Any non-semisimple pointed Hopf superalgebra of dimension $p^2$ is purely even.
\end{theorem}
\begin{proof}
Let $\HH$ be a non-semisimple pointed Hopf superalgebra of dimension $p^2$. Since $\widehat{\HH}$ is a non-semisimple pointed Hopf algebra of dimension $2p^2$, it is isomorphic to one of the following Hopf algebras (see \cite[Lemma~A.1]{AndNat01}).
\[ \AA(\tau,1,0),\; \AA(\tau,1,1),\; \AA(\omega,2r,0) \text{ or } \AA(\tau,2,0), \]
where $r\in\{1,\dots,p-1\}$, $\tau\in\kk$ is a primitive $p$-th root of unity and $\omega\in\kk$ is a fixed primitive $2p$-th root of unity. Since none of them admits a non-trivial super-form by Proposition~\ref{prp:A(w,j,mu)}, the claim follows.
\end{proof}

\begin{theorem}\label{thm:main3}
Let $p$ be an odd prime number, and let $\HH$ be a Hopf superalgebra of dimension $2p$ satisfying $\HH_{\bar1}\neq0$. If $\HH$ is non-semisimple and pointed, then $\HH$ is isomorphic to one of the Hopf superalgebras given in Table~\ref{tab:2p-dim}. Moreover, the Hopf superalgebras in Table~\ref{tab:2p-dim} are pairwise non-isomorphic.
\end{theorem}
\begin{proof}
Since $\widehat{\HH}$ is a non-semisimple pointed Hopf algebra of dimension $4p$, it is isomorphic to one of the following Hopf algebras (\cite[Lemma~A.1]{AndNat01}):
\[ \AA(-1,1,0),\; \AA(-1,1,1),\; \AA(\omega,p,0) \text{ or } \AA(-1,p,0), \]
where $\omega\in\kk$ is a fixed primitive $2p$-th root of unity. In the following, we describe the structure of the corresponding Hopf superalgebra explicitly. To do this, we let $\Gamma$ be the group $C_p=\langle g \mid g^p=1 \rangle$ of order $p$, and let $\chi$ be the character defined by $\chi(g):=\omega^2$.

By Theorem~\ref{thm:gen}, each coinvariant subalgebra is generated by $c^2$ and $x$. The isomorphisms corresponding to Table~\ref{tab:2p-dim} are
\[
\begin{aligned}
\HH_{2p}^{(1)}&=\AA(-1,p,0)^{\coo(c^p,\alpha)}
 \cong\AU(\Gamma,(1,\unit,0;1)),\\
\HH_{2p}^{(2)}&=\AA(-1,1,0)^{\coo(c^p,\alpha)}
 \cong\AU(\Gamma,(g^{(p+1)/2},\unit,0;1)),\\
\HH_{2p}^{(3)}&=\AA(\omega,p,0)^{\coo(c^p,\alpha)}
 \cong\AU(\Gamma,(1,\chi,0;1)),\\
\HH_{2p}^{(4)}&=\AA(-1,1,1)^{\coo(c^p,\alpha)}
 \cong\AU(\Gamma,(g^{(p+1)/2},\unit,1;1)).
\end{aligned}
\]
In each case the isomorphism from $\AU(\Gamma,\DD)$ to the coinvariant subalgebra sends $g$ to $c^2$ and $z$ to $x$.
\end{proof}
\begin{remark}
By definition, we have $\AU(\Gamma,(1,\chi,0;1))=\langle g,z \mid g^p=1,z^2=0,gz=\omega^2 zg\rangle$, where $g$ is group-like and $z$ is odd primitive. Suppose that $\zeta_p\in\kk$ is a primitive $p$-th root of unity. Then one easily sees that
\[ \AU(\Gamma,(1,\chi,0;1)) \cong \kk\langle g,z \mid g^p=1,z^2=0,gz=\zeta_p zg\rangle, \]
where $g$ is group-like and $z$ is odd primitive.
\end{remark}

It is easy to see that the exterior superalgebra $\bigwedge \kk$ is the only (up to isomorphism) Hopf superalgebra of dimension $2$ whose odd part is non-zero. By this observation and Theorems~\ref{thm:main1}, \ref{thm:main2} and \ref{thm:main3}, to complete the classification of non-semisimple pointed Hopf superalgebras $\HH$ (with $\HH_{\bar1}\neq0$) of dimension up to $10$, it remains to address the cases of $\dim(\HH)=4$ and $\dim(\HH) = 8$.

\subsection{Dimension $4$} \label{subsec:four-dim}
In this section, we show the following.
\begin{theorem} \label{prp:4-dim}
Let $\HH$ be a Hopf superalgebra of dimension $4$ satisfying $\HH_{\bar1}\neq0$. If $\HH$ is non-semisimple and pointed, then $\HH$ is isomorphic to one of the Hopf superalgebras given in Table~\ref{tab:4-dim}. Moreover, the Hopf superalgebras in Table~\ref{tab:4-dim} are pairwise non-isomorphic.
\end{theorem}

While this has indeed been proven in \cite{ShiWak23}, below we give an alternative proof of Theorem~\ref{prp:4-dim} by demonstrating a more conceptual approach using our Theorem~\ref{thm:gen}.

The classification of non-semisimple pointed Hopf algebras of dimension $8$ ($=2\times 4$) was carried out by \c{S}tefan~\cite{Ste99}. According to this result, such a Hopf algebra is isomorphic to one of the following Hopf algebras:
\begin{itemize}
\item $A_{C_2}:=\kk\langle c,x_1,x_2 \mid c^2=1, cx_i=-x_ic, x_ix_j=-x_jx_i\;(i,j\in\{1,2\}) \rangle$, where $c$ is group-like and $x_1,x_2$ are $c$-skew primitive.
\item $A_{C_2\times C_2}:=\kk\langle c,d,x \mid c^2=d^2=1,cd=dc,cx=-xc,dx=-xd,x^2=0\rangle$, where $c,d$ are group-like and $x$ is $c$-skew primitive.
\item $A'_{C_4}:=\kk\langle c,x \mid c^4=1,cx=-xc,x^2=0\rangle$, where $c$ is group-like and $x$ is $c$-skew primitive.
\item $A''_{C_4}:=\kk\langle c,x \mid c^4=1,cx=\zeta_4xc,x^2=0 \rangle$, where $c$ is group-like and $x$ is $c^2$-skew primitive.
\item $\kk\langle c,x \mid c^4=1,cx=-xc,x^2=c^2-1\rangle$, where $c$ is group-like and $x$ is $c$-skew primitive.
\end{itemize}

By Proposition~\ref{prp:super-forms}, we see that the Hopf algebras $A_{C_2}$ and $A_{C_2\times C_2}$ are the only ones which might have super-forms. Note that $\G(A_{C_2})\cong \Z_2$ and $\G(A_{C_2\times C_2})\cong \Z_2\times \Z_2$, so both groups are abelian.

\subsubsection{The $A_{C_2}$ case} \label{subsubsec:A_C_2}
First, we treat $A^{(1)}:=A_{C_2}$. By definition, the group-like element $c\in A^{(1)}$ is of order two. One sees that $\alpha\in \G((A^{(1)})^*)$ defined by $\alpha(c)=-1,\alpha(x_1)=\alpha(x_2)=0$ is the only algebra map of order two. Thus, we have $\AD{A^{(1)}}=\{(c,\alpha)\}$. Moreover, we see that $\SD{A^{(1)}}=\{(c,\alpha)\}$. By Theorem~\ref{thm:gen}, the coinvariant subalgebra $\HH_4^{(1)}:=(A^{(1)})^{\coo(c,\alpha)}$ is generated by $x_1$ and $x_2$. Also, $x_1$ and $x_2$ are odd primitive. Thus, we get the following result.
\begin{proposition} \label{prp:4-dim-1}
The exterior superalgebra $\bigwedge \kk^2$ is the only Hopf superalgebra with nonzero odd part whose bosonization is isomorphic to $A^{(1)}(=A_{C_2})$.
\end{proposition}

\subsubsection{The $A_{C_2\times C_2}$ case} \label{subsubsec:A_C_2C_2}
Next, we treat $A^{(2)}:=A_{C_2\times C_2}$. Both $\G(A^{(2)})$ and $\G((A^{(2)})^*)$ are isomorphic to $\Z_2\times \Z_2$ and are given by
\[ \G(A^{(2)}) = \{1,c,d,cd\} \quad\text{and}\quad \G((A^{(2)})^*) = \{\varepsilon,\alpha_1,\alpha_2,\alpha_3:=\alpha_1\alpha_2\}, \]
where $\alpha_1$ and $\alpha_2$ are algebra maps $A^{(2)}\to\kk$ determined by $\alpha_1(c)=-1$, $\alpha_1(d)=1$, $\alpha_2(c)=1$, $\alpha_2(d)=-1$ and $\alpha_1(x)=\alpha_2(x)=0$. Thus, the set of all admissible data for $A^{(2)}$ is
\[ \AD{A^{(2)}} = \{(c,\alpha_1),(cd,\alpha_1),(d,\alpha_2),(cd,\alpha_2),(c,\alpha_3),(d,\alpha_3)\}. \]
By \eqref{eq:SD}, we get $\SD{A^{(2)}} = \{ (c,\alpha_1),(c,\alpha_3),(d,\alpha_3) \}$. The following is easy to see.
\begin{lemma} \label{prp:A_{C_2,C_2}-autom}
We have $\HopfAuto(A^{(2)})=\{\varphi_u\}_{u\in\kk^\times}$, where $\varphi_u$ is determined by $\varphi_u|_{\G(A^{(2)})}=\id$ and $\varphi_u(x)=ux$.
\end{lemma}

By Lemma~\ref{prp:A_{C_2,C_2}-autom}, a complete set of representatives of $\SD{A^{(2)}}/\!\sim$ is given by $\{ (c,\alpha_1),(c,\alpha_3),(d,\alpha_3) \}$. In the following, we determine the structure of each coinvariant subalgebra of $A^{(2)}$ using Theorem~\ref{thm:gen}. To do this, we let $\Gamma$ be the group $C_2=\langle g \mid g^2=1 \rangle$ of order two, and let $\chi$ be the non-trivial character of $C_2$.
\begin{itemize}
\item Set $\HH_4^{(4)}:=(A^{(2)})^{\coo(c,\alpha_1)}$. As an algebra, $\HH_4^{(4)}$ is generated by $\{d,x\}$, and $x$ is odd primitive. Since $dx=-xd$, the assignment $g\mapsto d$, $z\mapsto x$ gives a Hopf superalgebra isomorphism $\AU(\Gamma,(1,\chi,0;1))\cong\HH_4^{(4)}$.
\item Set $\HH_4^{(2)}:=(A^{(2)})^{\coo(c,\alpha_3)}$. As an algebra, $\HH_4^{(2)}$ is generated by $\{cd,x\}$, and $x$ is odd primitive. Since $cd$ commutes with $x$, the assignment $g\mapsto cd$, $z\mapsto x$ gives a Hopf superalgebra isomorphism $\AU(\Gamma,(1,\unit,0;1))\cong\HH_4^{(2)}$.
\item Set $\HH_4^{(3)}:=(A^{(2)})^{\coo(d,\alpha_3)}$. As an algebra, $\HH_4^{(3)}$ is generated by $\{cd,x\}$, and $x$ is odd $cd$-skew primitive. The assignment $g\mapsto cd$, $z\mapsto x$ gives a Hopf superalgebra isomorphism $\AU(\Gamma,(g,\unit,0;1))\cong\HH_4^{(3)}$.
\end{itemize}
One sees $\HH_4^{(2)}\cong \kk C_2\otimes \bigwedge\kk$.

Therefore, we get the following result.
\begin{proposition}\label{prp:4-dim-2}
The non-trivial super-forms of $A^{(2)}(=A_{C_2\times C_2})$ are the Hopf superalgebras $\AU(\Gamma,\DD)$ with
\[
\DD\in\{(1,\unit,0;1), (1,\chi,0;1), (g,\unit,0;1)\}.
\]
They are pairwise non-isomorphic.
\end{proposition}

By Propositions~\ref{prp:4-dim-1} and \ref{prp:4-dim-2}, the proof of Theorem~\ref{prp:4-dim} is complete.

\subsection{Dimension $8$} \label{subsec:eight-dim}
In this section, we show the following.
\begin{theorem} \label{prp:8-dim}
Let $\HH$ be a Hopf superalgebra of dimension $8$ satisfying $\HH_{\bar1}\neq0$. If $\HH$ is non-semisimple and pointed, then $\HH$ is isomorphic to one of the Hopf superalgebras given in Table~\ref{tab:8-dim}. Moreover, the Hopf superalgebras in Table~\ref{tab:8-dim} are pairwise non-isomorphic.
\end{theorem}

The classification of non-semisimple pointed Hopf algebras of dimension $16$ ($=2\times 8$) was carried out by Caenepeel, D\u{a}sc\u{a}lescu and Raianu~\cite{CaeDasRai00}. According to their result, there are $29$ such Hopf algebras. By Proposition~\ref{prp:super-forms}, we see that among those $29$ Hopf algebras, the ones listed in Table~\ref{tab:CDR} are candidates for having super-forms. In the following, we classify Hopf superalgebras $\HH$ such that $\HH_{\bar1}\neq0$ and $\widehat{\HH}\cong A^{(j)}$ for each $j=1,\dots,14$.
\begin{table}[h]
\begin{tabular}{c|l}
Hopf algebras & Notation in \cite{CaeDasRai00} \\ \hline
$A^{(1)}$ & $H(C_2,(2,2,2),(c^*,c^*,c^*),(c,c,c),(0,0,0))$ \\
$A^{(2)}$ & $H(C_2\times C_2,(2,2),(c^*,c^*),(c,c),(0,0))$ \\
$A^{(3)}$ & $H(C_2\times C_2,(2,2),(c^*,c^*d^*),(c,c),(0,0))$ \\
$A^{(4)}$ & $H(C_2\times C_2,(2,2),(c^*,d^*),(c,d),(0,0))$ \\
$A^{(5)}$ & $H(C_2\times C_2,(2,2),(c^*d^*,c^*d^*),(c,d),(0,0))$ \\
$A^{(6)}$ & $H(C_2\times C_2,(2,2),(c^*d^*,c^*d^*),(c,d),(0,0),\left(\begin{smallmatrix}0&1\\1&0\end{smallmatrix}\right))$ \\
$A^{(7)}$ & $H(C_2\times C_2\times C_2,2,c^*,c,0)$ \\
$A^{(8)}$ & $H(C_4\times C_2,2,(c^*)^2,c,0)$ \\
$A^{(9)}$ & $H(C_4\times C_2,2,d^*,cd,0)$ \\
$A^{(10)}$ & $H(C_4\times C_2,2,c^*,c^2,0)$ \\
$A^{(11)}$ & $H(C_4\times C_2,2,d^*,d,0)$ \\
$A^{(12)}$ & $H(C_4\times C_2,2,c^*,c^2d,0)$ \\
$A^{(13)}$ & $H(C_4\times C_2,2,(c^*)^2,c,1)$ \\
$A^{(14)}$ & $H(C_4\times C_2,2,d^*,cd,1)$
\end{tabular}
\caption{The Hopf algebras considered in Section~\ref{subsec:eight-dim}.}\label{tab:CDR}
\end{table}

\subsubsection{The $A^{(1)}$ case} \label{subsubsec:A^1}
The Hopf algebra $A^{(1)}$ is given by
\[ A^{(1)}=\kk\langle c,x_1,x_2,x_3 \mid c^2=1,x_ix_j=-x_jx_i,cx_i=-x_ic\; (i,j\in\{1,2,3\}) \rangle,\]
where $c$ is group-like, and $x_1,x_2$ and $x_3$ are $c$-skew primitive.

One easily sees that $\AD{A^{(1)}}=\SD{A^{(1)}}=\{(c,\alpha)\}$, where $\alpha(c)=-1$ and $\alpha(x_i)=0$ ($i\in\{1,2,3\}$). By Theorem~\ref{thm:gen}, the coinvariant subalgebra $\HH_8^{(1)}:=(A^{(1)})^{\coo(c,\alpha)}$ is generated by $\{x_1,x_2,x_3\}$. Also, $x_1$, $x_2$ and $x_3$ are odd primitive. The above argument shows the following result.
\begin{proposition}
The exterior superalgebra $\bigwedge \kk^3$ is the only Hopf superalgebra with nonzero odd part whose bosonization is isomorphic to $A^{(1)}$.
\end{proposition}

\subsubsection{The $A^{(2)}$ case} \label{subsubsec:A^2}
The Hopf algebra $A^{(2)}$ is given by
\[A^{(2)}=\kk\left\langle c,d,x_1,x_2 \mathrel{}\middle|\mathrel{} \begin{array}{c}c^2=d^2=1,cd=dc,x_ix_j=-x_jx_i,\\
cx_i=-x_ic,dx_i=x_id\; (i,j\in\{1,2\}) \end{array}\right\rangle,\]
where $c$ and $d$ are group-like, and $x_1$ and $x_2$ are $c$-skew primitive.

First, note that both $\G(A^{(2)})$ and $\G((A^{(2)})^*)$ are isomorphic to $\Z_2\times \Z_2$ and are given by
\[ \G(A^{(2)}) = \{1,c,d,cd\} \quad\text{and}\quad \G((A^{(2)})^*) = \{\varepsilon,\alpha_1,\alpha_2,\alpha_3:=\alpha_1\alpha_2\}, \]
where $\alpha_1$ and $\alpha_2$ are algebra maps $A^{(2)}\to\kk$ determined by $\alpha_1(c)=-1$, $\alpha_1(d)=1$, $\alpha_2(c)=1$, $\alpha_2(d)=-1$ and $\alpha_1(x_i)=\alpha_2(x_i)=0$ for $i\in\{1,2\}$. Thus, we get
\[ \AD{A^{(2)}} = \{ (c,\alpha_1), (cd,\alpha_1), (d,\alpha_2), (cd,\alpha_2), (c,\alpha_3), (d,\alpha_3) \}. \]
By \eqref{eq:SD}, we have $\SD{A^{(2)}} = \{ (c,\alpha_1), (cd,\alpha_1), (c,\alpha_3) \}$. Let $\GL_2(\kk)$ denote the general linear group of degree two over $\kk$. The following is easy to see.
\begin{lemma} \label{prp:A_2-autom}
We have $\HopfAuto(A^{(2)})=\{\varphi_P\}_{P\in\GL_2(\kk)}$, where $\varphi_P$ is determined by $\varphi_P|_{\G(A^{(2)})}=\id$ and $\varphi_P(x_i)=p_{i,1}x_1+p_{i,2}x_2$ for $i\in\{1,2\}$ with $p_{i,j}$ the $(i,j)$-entry of the matrix $P$.
\end{lemma}
Hence, we get $\#(\SD{A^{(2)}})=\#(\SD{A^{(2)}}/\!\sim)$. Let $\Gamma$ be the group $C_2=\langle g \mid g^2=1 \rangle$ of order two, and let $\chi$ be the non-trivial character.
\begin{proposition}
The Hopf superalgebras $\AU(\Gamma,\DD)$ with \[\DD\in\{((1,\unit,0;1),(1,\unit,0;1)),\; ((g,\unit,0;1),(g,\unit,0;1)),\; ((1,\chi,0;1),(1,\chi,0;1))\}\] are the only ones with nonzero odd part whose bosonization is isomorphic to $A^{(2)}$. Moreover, these are pairwise non-isomorphic.
\end{proposition}
\begin{proof}
In the following, we determine the structure of each coinvariant subalgebra of $A^{(2)}$ using Theorem~\ref{thm:gen} one by one.
\begin{itemize}
\item The Hopf superalgebra $\HH_8^{(2)}:=(A^{(2)})^{\coo(c,\alpha_1)}$ is generated by $d,x_1,x_2$. As an element in $\HH_8^{(2)}$, $x_i$ is odd primitive ($i\in\{1,2\}$). Thus, $g\mapsto d,z_i\mapsto x_i$ gives a Hopf superalgebra isomorphism $\AU(\Gamma,((1,\unit,0;1),(1,\unit,0;1))) \cong \HH_8^{(2)}$.
\item The Hopf superalgebra $\HH_8^{(5)}:=(A^{(2)})^{\coo(cd,\alpha_1)}$ is generated by $d,x_1,x_2$. As an element in $\HH_8^{(5)}$, $x_i$ is odd $d$-skew primitive ($i\in\{1,2\}$). Thus, the assignment $g\mapsto d,z_i\mapsto x_i$ gives a Hopf superalgebra isomorphism $\AU(\Gamma,((g,\unit,0;1),(g,\unit,0;1)))\cong\HH_8^{(5)}$.
\item Set $\HH_8^{(6)}:=(A^{(2)})^{\coo(c,\alpha_3)}$. As an element in $\HH_8^{(6)}$, $x_i$ is odd primitive ($i\in\{1,2\}$). Thus, the assignment $g\mapsto cd,z_i\mapsto x_i$ gives a Hopf superalgebra isomorphism $\AU(\Gamma,((1,\chi,0;1),(1,\chi,0;1))) \cong \HH_8^{(6)}$.
\end{itemize}
This completes the proof.
\end{proof}
One easily sees $\HH_8^{(2)}\cong \kk C_2\otimes \bigwedge\kk^2$.

\subsubsection{The $A^{(3)}$ case} \label{subsubsec:A^3}
The Hopf algebra $A^{(3)}$ is given by
\[ A^{(3)}=\kk\left\langle c,d,x_1,x_2 \mathrel{}\middle|\mathrel{}\begin{array}{c} c^2=d^2=1,cd=dc,x_ix_j=-x_jx_i,\\
cx_i=-x_ic,dx_i=-(-1)^ix_id\; (i,j\in\{1,2\})\end{array} \right\rangle, \]
where $c$ and $d$ are group-like, and $x_1$ and $x_2$ are $c$-skew primitive.

First, note that both $\G(A^{(3)})$ and $\G((A^{(3)})^*)$ are isomorphic to $\Z_2\times \Z_2$ and are given by
\[ \G(A^{(3)}) = \{1,c,d,cd\} \quad\text{and}\quad \G((A^{(3)})^*) = \{\varepsilon,\alpha_1,\alpha_2,\alpha_3:=\alpha_1\alpha_2\}, \]
where $\alpha_1$ and $\alpha_2$ are algebra maps $A^{(3)}\to\kk$ determined by $\alpha_1(c)=-1$, $\alpha_1(d)=1$, $\alpha_2(c)=1$, $\alpha_2(d)=-1$ and $\alpha_1(x_i)=\alpha_2(x_i)=0$ for $i\in\{1,2\}$. Thus, we get
\[ \AD{A^{(3)}} = \{ (c,\alpha_1), (cd,\alpha_1), (d,\alpha_2), (cd,\alpha_2), (c,\alpha_3), (d,\alpha_3) \}. \]
By \eqref{eq:SD}, we have $\SD{A^{(3)}} = \{ (c,\alpha_1), (c,\alpha_3) \}$. The following is easy to see.
\begin{lemma}
We have $\HopfAuto(A^{(3)}) = \langle \tau, \varphi_{a_1,a_2} \mid a_1, a_2 \in \kk^{\times} \rangle$, where
\[ \begin{gathered}\tau(c) = c,\quad \tau(d) = cd,\quad \tau(x_1) = x_2,\quad \tau(x_2) = x_1,\\
\varphi_{a_1, a_2}(c) = c,\quad \varphi_{a_1, a_2}(d) = d,\quad \varphi_{a_1, a_2}(x_i) = a_i x_i\quad(i\in\{1,2\}). \end{gathered} \]
\end{lemma}

Thus, we get $\SD{A^{(3)}}/\!\sim = \{ [(c,\alpha_1)] \}$. Let $\Gamma$ be the group $C_2=\langle g\mid g^2=1 \rangle$ of order two, and let $\chi$ be the non-trivial character.
\begin{proposition}
The Hopf superalgebra $\AU(\Gamma,((1,\unit,0;1),(1,\chi,0;1)))$ is the only one with nonzero odd part whose bosonization is isomorphic to $A^{(3)}$.
\end{proposition}
\begin{proof}
The coinvariant subalgebra $\HH_8^{(4)}:=(A^{(3)})^{\coo(c,\alpha_1)}$ is generated by $d,x_1,x_2$. By Theorem~\ref{thm:gen}, $x_1$ and $x_2$ are odd primitive elements of $\HH_8^{(4)}$. Thus, the assignment $g\mapsto d,z_i\mapsto x_i$ gives a Hopf superalgebra isomorphism
\[ \AU(\Gamma,((1,\unit,0;1),(1,\chi,0;1))) \cong \HH_8^{(4)}. \]
\end{proof}

\subsubsection{The $A^{(4)}$ case} \label{subsubsec:A^4}
The Hopf algebra $A^{(4)}$ is given by
\[ A^{(4)}=\kk\left\langle c,d,x_1,x_2 \mathrel{}\middle|\mathrel{} \begin{array}{c} c^2=d^2=1,cd=dc,x_ix_j=x_jx_i,x_i^2=0,\\
cx_i=(-1)^ix_ic,dx_i=-(-1)^ix_id\; (i,j\in\{1,2\}) \end{array}\right\rangle, \]
where $c$ and $d$ are group-like, and $x_1$ (resp.~$x_2$) is $c$-skew (resp.~$d$-skew) primitive.

First, note that both $\G(A^{(4)})$ and $\G((A^{(4)})^*)$ are isomorphic to $\Z_2\times \Z_2$ and are given by
\[ \G(A^{(4)}) = \{1,c,d,cd\} \quad\text{and}\quad \G((A^{(4)})^*) = \{\varepsilon,\alpha_1,\alpha_2,\alpha_3:=\alpha_1\alpha_2\}, \]
where $\alpha_1$ and $\alpha_2$ are algebra maps $A^{(4)}\to\kk$ determined by $\alpha_1(c)=-1$, $\alpha_1(d)=1$, $\alpha_2(c)=1$, $\alpha_2(d)=-1$ and $\alpha_1(x_i)=\alpha_2(x_i)=0$ for $i\in\{1,2\}$. Thus, we get
\[ \AD{A^{(4)}} = \{ (c,\alpha_1), (cd,\alpha_1), (d,\alpha_2), (cd,\alpha_2), (c,\alpha_3), (d,\alpha_3) \}. \]
By \eqref{eq:SD}, we have $\SD{A^{(4)}} = \{ (c,\alpha_1), (d,\alpha_2) \}$. The following is easy to see.
\begin{lemma}
We have $\HopfAuto(A^{(4)}) = \langle \tau, \varphi_{a_1,a_2} \mid a_1, a_2 \in \kk^{\times} \rangle$, where
\[ \begin{gathered}\tau(c) = d,\quad \tau(d) = c,\quad \tau(x_1) = x_2,\quad \tau(x_2) = x_1,\\
\varphi_{a_1, a_2}(c) = c,\quad \varphi_{a_1, a_2}(d) = d,\quad \varphi_{a_1, a_2}(x_i) = a_i x_i\quad(i\in\{1,2\}). \end{gathered} \]
\end{lemma}

Thus, we get $(c,\alpha_1)\sim(d,\alpha_2)$. Thus, we get $\SD{A^{(4)}}/\!\sim = \{ [(c,\alpha_1)] \}$. Let $\Gamma$ be the group $C_2=\langle g \mid g^2=1 \rangle$ of order two, and let $\chi$ be the non-trivial character.
\begin{proposition}
The Hopf superalgebra $\AU(\Gamma,((g,\chi,0;0),(1,\unit,0;1)))$ is the only one with nonzero odd part whose bosonization is isomorphic to $A^{(4)}$.
\end{proposition}
\begin{proof}
The coinvariant subalgebra $\HH_8^{(7)}:=(A^{(4)})^{\coo(c,\alpha_1)}$ is generated by $d,x_1,x_2$. By Theorem~\ref{thm:gen}, in $\HH_8^{(7)}$, $x_1$ is odd primitive and $x_2$ is even $d$-skew primitive. Thus, the assignment $g\mapsto d,z_1\mapsto x_2, z_2\mapsto x_1$ gives a Hopf superalgebra isomorphism $\AU(\Gamma,((g,\chi,0;0),(1,\unit,0;1))) \cong \HH_8^{(7)}$.
\end{proof}
One sees $\HH_8^{(7)}\cong T_4(-1)\otimes\bigwedge \kk$.

\subsubsection{The $A^{(5)}$ case} \label{subsubsec:A^5}
The Hopf algebra $A^{(5)}$ is given by
\[ A^{(5)}=\kk\left\langle c,d,x_1,x_2 \mathrel{}\middle|\mathrel{}\begin{array}{c} c^2=d^2=1,cd=dc,x_ix_j=-x_jx_i,\\
cx_i=-x_ic,dx_i=-x_id\; (i,j\in\{1,2\}) \end{array}\right\rangle, \]
where $c$ and $d$ are group-like, and $x_1$ (resp.~$x_2$) is $c$-skew (resp.~$d$-skew) primitive.

First, note that both $\G(A^{(5)})$ and $\G((A^{(5)})^*)$ are isomorphic to $\Z_2\times \Z_2$ and are given by
\[ \G(A^{(5)}) = \{1,c,d,cd\} \quad\text{and}\quad \G((A^{(5)})^*) = \{\varepsilon,\alpha_1,\alpha_2,\alpha_3:=\alpha_1\alpha_2\}, \]
where $\alpha_1$ and $\alpha_2$ are algebra maps $A^{(5)}\to\kk$ determined by $\alpha_1(c)=-1$, $\alpha_1(d)=1$, $\alpha_2(c)=1$, $\alpha_2(d)=-1$ and $\alpha_1(x_i)=\alpha_2(x_i)=0$ for $i\in\{1,2\}$. Thus, we get
\[ \AD{A^{(5)}} = \{ (c,\alpha_1), (cd,\alpha_1), (d,\alpha_2), (cd,\alpha_2), (c,\alpha_3), (d,\alpha_3) \}. \]
By \eqref{eq:SD}, we have $\SD{A^{(5)}} = \{ (c,\alpha_3), (d,\alpha_3) \}$. The following is easy to see.
\begin{lemma}
We have $\HopfAuto(A^{(5)}) = \langle \tau, \varphi_{a_1,a_2} \mid a_1, a_2 \in \kk^{\times} \rangle$, where
\[ \begin{gathered}\tau(c) = d,\quad \tau(d) = c,\quad \tau(x_1) = x_2,\quad \tau(x_2) = x_1,\\
\varphi_{a_1, a_2}(c) = c,\quad \varphi_{a_1, a_2}(d) = d,\quad \varphi_{a_1, a_2}(x_i) = a_i x_i\quad(i\in\{1,2\}). \end{gathered} \]
\end{lemma}

Thus, we get $(c,\alpha_3)\sim(d,\alpha_3)$. Hence, we get $\SD{A^{(5)}}/\!\sim = \{ [(c,\alpha_3)] \}$. Let $\Gamma$ be the group $C_2=\langle g\mid g^2=1\rangle$ of order two.
\begin{proposition}
The Hopf superalgebra $\AU(\Gamma,((1,\unit,0;1),(g,\unit,0;1)))$ is the only one with nonzero odd part whose bosonization is isomorphic to $A^{(5)}$.
\end{proposition}
\begin{proof}
The coinvariant subalgebra $\HH_8^{(3)}:=(A^{(5)})^{\coo(c,\alpha_3)}$ is generated by $cd,x_1,x_2$. By Theorem~\ref{thm:gen}, in $\HH_8^{(3)}$, $x_1$ is odd primitive and $x_2$ is odd $cd$-skew primitive. Thus, the assignment $g\mapsto cd,z_i\mapsto x_i$ gives a Hopf superalgebra isomorphism $\AU(\Gamma,((1,\unit,0;1),(g,\unit,0;1))) \cong \HH_8^{(3)}$.
\end{proof}

\subsubsection{The $A^{(6)}$ case} \label{subsubsec:A^6}
The Hopf algebra $A^{(6)}$ is given by
\[ A^{(6)}=\kk\left\langle c,d,x_1,x_2 \mathrel{}\middle|\mathrel{} \begin{array}{c} c^2=d^2=1,cd=dc,x_2x_1=-x_1x_2+cd-1,x_i^2=0,\\
cx_i=-x_ic,dx_i=-x_id\; (i\in\{1,2\}) \end{array}\right\rangle, \]
where $c$ and $d$ are group-like, and $x_1$ (resp.~$x_2$) is $c$-skew (resp.~$d$-skew) primitive.

The group-like elements of $(A^{(6)})^*$ are $\varepsilon$ and $\alpha$, where
\[
\alpha(c)=\alpha(d)=-1,\qquad \alpha(x_1)=\alpha(x_2)=0.
\]
By \eqref{eq:SD}, we have
\[
\AD{A^{(6)}}=\SD{A^{(6)}}=\{(c,\alpha),(d,\alpha)\}.
\]
The automorphism interchanging $c,d$ and $x_1,x_2$ interchanges these two super-data.
Thus they form a single equivalence class.

Let $\HH_8^{(19)}$ be the superalgebra generated by an even element $g$ and odd elements $x,y$ with relations
\begin{equation}\label{eq:H19-relations}
g^2=1,\qquad gx=xg,\qquad gy=yg,\qquad
x^2=y^2=0,\qquad xy+yx=1-g.
\end{equation}
Its Hopf superalgebra structure is given by
\begin{equation}\label{eq:H19-Hopf}
\begin{aligned}
\Delta(g)&=g\otimes g,&
\Delta(x)&=1\otimes x+x\otimes1,&
\Delta(y)&=g\otimes y+y\otimes1,\\
\varepsilon(g)&=1,&\varepsilon(x)&=0,&\varepsilon(y)&=0,\\
S(g)&=g,&S(x)&=-x,&S(y)&=-gy.
\end{aligned}
\end{equation}

\begin{proposition}\label{prop:A6-corrected}
The Hopf superalgebra $\HH_8^{(19)}$ is the unique non-trivial super-form of $A^{(6)}$, up to isomorphism.
It is not isomorphic to any of $\HH_8^{(1)},\ldots,\HH_8^{(18)}$.
\end{proposition}
\begin{proof}
Put $B=(A^{(6)})^{\coo(c,\alpha)}$.
By Theorem~\ref{thm:gen}, $B$ is generated by $cd,x_1,x_2$, where $cd$ is even, $x_1$ is odd primitive and $x_2$ is odd $cd$-skew primitive.
The assignment
\[
g\longmapsto cd,\qquad x\longmapsto x_1,\qquad y\longmapsto-x_2
\]
defines a surjective algebra map from the superalgebra with relations \eqref{eq:H19-relations} to $B$.
Those relations show that the superalgebra is spanned by the eight monomials $g^a x^b y^c$ with $a,b,c\in\{0,1\}$.
Since $\dim B=\dim A^{(6)}/2=8$, this map is an isomorphism.
The Hopf superalgebra structure of $B$ gives the formulas in \eqref{eq:H19-Hopf}.
By Proposition~\ref{prp:pt/ss/chev}, $B$ is pointed and non-semisimple, since $A^{(6)}$ has these properties.
The calculation of super-data above proves uniqueness.
Finally, $A^{(6)}$ is not isomorphic to any other Hopf algebra in Table~\ref{tab:CDR} by \cite{CaeDasRai00}.
Since the bosonizations of $\HH_8^{(1)},\ldots,\HH_8^{(18)}$ occur among those other Hopf algebras, the non-isomorphism assertion follows.
\end{proof}

\subsubsection{The $A^{(7)}$ case} \label{subsubsec:A^7}
The Hopf algebra $A^{(7)}$ is given by
\[ A^{(7)}=\kk\left\langle c,d,e,x \mathrel{}\middle|\mathrel{}\begin{array}{c} c^2=d^2=e^2=1,cd=dc,ce=ec,de=ed,x^2=0,\\
cx=-xc,dx=xd,ex=xe\end{array} \right\rangle, \]
where $c,d$ and $e$ are group-like, and $x$ is $c$-skew primitive.

First, note that both $\G(A^{(7)})$ and $\G((A^{(7)})^*)$ are isomorphic to $\Z_2\times \Z_2\times\Z_2$ and are given by
\[ \G(A^{(7)}) = \langle c,d,e\rangle \quad\text{and}\quad \G((A^{(7)})^*) = \langle\alpha_1,\alpha_2,\alpha_3\rangle, \]
where $\alpha_1$, $\alpha_2$ and $\alpha_3$ are algebra maps $A^{(7)}\to\kk$ determined by $\alpha_1(c)=-1$, $\alpha_1(d)=1$, $\alpha_1(e)=1$, $\alpha_2(c)=1$, $\alpha_2(d)=-1$, $\alpha_2(e)=1$, $\alpha_3(c)=1$, $\alpha_3(d)=1$, $\alpha_3(e)=-1$ and $\alpha_1(x)=\alpha_2(x)=\alpha_3(x)=0$. Set $\alpha_4:=\alpha_1\alpha_2$, $\alpha_5:=\alpha_1\alpha_3$, $\alpha_6:=\alpha_2\alpha_3$ and $\alpha_7:=\alpha_1\alpha_2\alpha_3$. By \eqref{eq:SD}, one sees that the set $\SD{A^{(7)}}$ is given as follows:
\[\begin{gathered} \{ (c,\alpha_1), (c,\alpha_4), (c,\alpha_5), (c,\alpha_7), (cd,\alpha_1), (cd,\alpha_5),\\ (ce,\alpha_1), (ce,\alpha_4), (cde,\alpha_1), (cde,\alpha_7)\}. \end{gathered}\]
The following is easy to see.
\begin{lemma}
We have $\HopfAuto(A^{(7)}) = \langle \sigma, \tau, \varphi_u \mid u \in \kk^{\times} \rangle$, where
\[ \begin{gathered}\sigma(c)=c,\quad \sigma(d)=de,\quad \sigma(e)=e,\quad \sigma(x)=x,\\
\tau(c) = c,\quad \tau(d) = e,\quad \tau(e) = d,\quad \tau(x) = x,\\
\varphi_u|_{\G(A^{(7)})}=\id,\quad \varphi_u(x)=ux. \end{gathered} \]
\end{lemma}

We get the following.
\begin{lemma}
The set $\{ (c,\alpha_1), (c,\alpha_4), (cd,\alpha_1), (cd,\alpha_5) \}$ is a complete set of representatives of $\SD{A^{(7)}}/\!\sim$.
\end{lemma}

Let $\Gamma$ be the group $C_2\times C_2=\langle g_1,g_2 \mid g_1^2=g_2^2=1, g_1g_2=g_2g_1 \rangle$, and let $\chi_i$ be the character defined by $\chi_i(g_j)=-(-1)^{i+j}$ $(i,j\in\{1,2\})$.
\begin{proposition}
The non-trivial super-forms of $A^{(7)}$ are the Hopf superalgebras $\AU(\Gamma,\DD)$ with
\[
\DD\in\{(1,\unit,0;1), (g_1,\unit,0;1), (1,\chi_1,0;1), (g_1,\chi_2,0;1)\}.
\]
They are pairwise non-isomorphic.
\end{proposition}
\begin{proof}
In the following, we determine the structure of each coinvariant subalgebra of $A^{(7)}$ using Theorem~\ref{thm:gen} one by one.
\begin{itemize}
\item The Hopf superalgebra $\HH_8^{(8)}:=(A^{(7)})^{\coo(c,\alpha_1)}$ is generated by $d,e,x$. As an element in $\HH_8^{(8)}$, $x$ is odd primitive. Thus, the assignment $g_1\mapsto d,g_2\mapsto e, z\mapsto x$ gives a Hopf superalgebra isomorphism $\AU(\Gamma,(1,\unit,0;1)) \cong \HH_8^{(8)}$.
\item The Hopf superalgebra $\HH_8^{(10)}:=(A^{(7)})^{\coo(c,\alpha_4)}$ is generated by $cd,e,x$. As an element in $\HH_8^{(10)}$, $x$ is odd primitive. Thus, $g_1\mapsto cd,g_2\mapsto e, z\mapsto x$ gives a Hopf superalgebra isomorphism $\AU(\Gamma,(1,\chi_1,0;1)) \cong \HH_8^{(10)}$.
\item The Hopf superalgebra $\HH_8^{(9)}:=(A^{(7)})^{\coo(cd,\alpha_1)}$ is generated by $d,e,x$. As an element in $\HH_8^{(9)}$, $x$ is odd $d$-skew primitive. Thus, the assignment $g_1\mapsto d,g_2\mapsto e, z\mapsto x$ gives a Hopf superalgebra isomorphism $\AU(\Gamma,(g_1,\unit,0;1)) \cong \HH_8^{(9)}$.
\item The Hopf superalgebra $\HH_8^{(11)}:=(A^{(7)})^{\coo(cd,\alpha_5)}$ is generated by $d,ce,x$. As an element in $\HH_8^{(11)}$, $x$ is odd $d$-skew primitive. Thus, the assignment $g_1\mapsto d,g_2\mapsto ce, z\mapsto x$ gives a Hopf superalgebra isomorphism $\AU(\Gamma,(g_1,\chi_2,0;1)) \cong \HH_8^{(11)}$.
\end{itemize}
This completes the proof.
\end{proof}

\subsubsection{The $A^{(8)}$ case} \label{subsubsec:A^8}
Let $\zeta_4\in\kk$ be a primitive fourth root of unity. The Hopf algebra $A^{(8)}$ is given by
\[ A^{(8)}=\kk\langle c,d,x \mid c^4=d^2=1,cd=dc,x^2=0, cx=-xc,dx=xd \rangle, \]
where $c$ and $d$ are group-like, and $x$ is $c$-skew primitive.

First, note that both $\G(A^{(8)})$ and $\G((A^{(8)})^*)$ are isomorphic to $\Z_4\times \Z_2$ and are given by
\[ \G(A^{(8)}) = \langle c,d \rangle \quad\text{and}\quad \G((A^{(8)})^*) = \langle\alpha,\beta\rangle, \]
where $\alpha$ and $\beta$ are algebra maps $A^{(8)}\to\kk$ determined by $\alpha(c)=\zeta_4,\alpha(d)=1,\alpha(x)=0$ and $\beta(c)=1,\beta(d)=-1,\beta(x)=0$, respectively. Set $\alpha_1:=\alpha^2,\alpha_2:=\beta,\alpha_3:=\alpha^2\beta$. We get
\[ \AD{A^{(8)}} = \{ (d,\alpha_2), (c^2d,\alpha_2), (d,\alpha_3), (c^2d,\alpha_3) \}. \]
By \eqref{eq:SD}, one easily sees that $\SD{A^{(8)}} = \varnothing$. Thus, we have the following result.
\begin{proposition}
The Hopf algebra $A^{(8)}$ has no non-trivial super-form.
\end{proposition}

\subsubsection{The $A^{(9)}$ case} \label{subsubsec:A^9}
Let $\zeta_4\in\kk$ be a primitive fourth root of unity. The Hopf algebra $A^{(9)}$ is given by
\[ A^{(9)}=\kk\langle c,d,x \mid c^4=d^2=1,cd=dc,x^2=0, cx=xc,dx=-xd \rangle, \]
where $c$ and $d$ are group-like, and $x$ is $cd$-skew primitive.

First, note that both $\G(A^{(9)})$ and $\G((A^{(9)})^*)$ are isomorphic to $\Z_4\times \Z_2$ and are given by
\[ \G(A^{(9)}) = \langle c,d \rangle \quad\text{and}\quad \G((A^{(9)})^*) = \langle\alpha,\beta\rangle, \]
where $\alpha$ and $\beta$ are algebra maps $A^{(9)}\to\kk$ determined by $\alpha(c)=\zeta_4,\alpha(d)=1,\alpha(x)=0$ and $\beta(c)=1,\beta(d)=-1,\beta(x)=0$, respectively. Set $\alpha_1:=\alpha^2,\alpha_2:=\beta,\alpha_3:=\alpha^2\beta$. We get
\[ \AD{A^{(9)}} = \{ (d,\alpha_2), (c^2d,\alpha_2), (d,\alpha_3), (c^2d,\alpha_3) \}. \]
By \eqref{eq:SD}, we have $\SD{A^{(9)}} = \{ (d,\alpha_2), (c^2d,\alpha_2)\}$. The following is easy to see.
\begin{lemma} \label{prp:A_9-autom}
We have $\HopfAuto(A^{(9)}) = \langle \sigma, \varphi_u \mid u\in\kk^\times \rangle$, where $\sigma(c)=c^3,\sigma(d)=c^2d,\sigma(x)=x$, $\varphi_u|_{\G(A^{(9)})}=\id$ and $\varphi_u(x)=ux$.
\end{lemma}

Hence, we get $\SD{A^{(9)}}/\!\sim =\{ [(d,\alpha_2)] \}$. Let $\Gamma$ be the group $C_4=\langle g\mid g^4=1\rangle$ of order four.
\begin{proposition}
The Hopf superalgebra $\AU(\Gamma,(g,\unit,0;1))$ is the only one with nonzero odd part whose bosonization is isomorphic to $A^{(9)}$.
\end{proposition}
\begin{proof}
The Hopf superalgebra $\HH_8^{(13)}:=(A^{(9)})^{\coo(d,\alpha_2)}$ is generated by $c,x$. As an element in $\HH_8^{(13)}$, $x$ is odd $c$-skew primitive. Thus, the assignment $g\mapsto c, z\mapsto x$ gives a Hopf superalgebra isomorphism $\AU(\Gamma,(g,\unit,0;1)) \cong \HH_8^{(13)}$.
\end{proof}

\subsubsection{The $A^{(10)}$ case} \label{subsubsec:A^10}
Let $\zeta_4\in\kk$ be a primitive fourth root of unity. The Hopf algebra $A^{(10)}$ is given by
\[ A^{(10)}=\kk\langle c,d,x \mid c^4=d^2=1,cd=dc,x^2=0, cx=\zeta_4xc,dx=xd \rangle,\]
where $c$ and $d$ are group-like, and $x$ is $c^2$-skew primitive.

First, note that both $\G(A^{(10)})$ and $\G((A^{(10)})^*)$ are isomorphic to $\Z_4\times \Z_2$ and are given by
\[ \G(A^{(10)}) = \langle c,d \rangle \quad\text{and}\quad \G((A^{(10)})^*) = \langle\alpha,\beta\rangle, \]
where $\alpha$ and $\beta$ are algebra maps $A^{(10)}\to\kk$ determined by $\alpha(c)=\zeta_4,\alpha(d)=1,\alpha(x)=0$ and $\beta(c)=1,\beta(d)=-1,\beta(x)=0$, respectively. Set $\alpha_1:=\alpha^2,\alpha_2:=\beta,\alpha_3:=\alpha^2\beta$. We get
\[ \AD{A^{(10)}} = \{ (d,\alpha_2), (c^2d,\alpha_2), (d,\alpha_3), (c^2d,\alpha_3) \}. \]
Since $d\in Z(A^{(10)})$ and $\alpha_2(c^2)=\alpha_3(c^2)=1$, we have $\SD{A^{(10)}} =\varnothing$ by \eqref{eq:SD}. Thus, we get the following result.
\begin{proposition}
The Hopf algebra $A^{(10)}$ has no non-trivial super-form.
\end{proposition}

\subsubsection{The $A^{(11)}$ case} \label{subsubsec:A^11}
Let $\zeta_4\in\kk$ be a primitive fourth root of unity. The Hopf algebra $A^{(11)}$ is given by
\[ A^{(11)}=\kk\langle c,d,x \mid c^4=d^2=1,cd=dc,x^2=0, cx=xc,dx=-xd \rangle, \]
where $c$ and $d$ are group-like, and $x$ is $d$-skew primitive.

First, note that both $\G(A^{(11)})$ and $\G((A^{(11)})^*)$ are isomorphic to $\Z_4\times \Z_2$ and are given by
\[ \G(A^{(11)}) = \langle c,d \rangle \quad\text{and}\quad \G((A^{(11)})^*) = \langle\alpha,\beta\rangle, \]
where $\alpha$ and $\beta$ are algebra maps $A^{(11)}\to\kk$ determined by $\alpha(c)=\zeta_4,\alpha(d)=1,\alpha(x)=0$ and $\beta(c)=1,\beta(d)=-1,\beta(x)=0$, respectively. Set $\alpha_1:=\alpha^2,\alpha_2:=\beta,\alpha_3:=\alpha^2\beta$. We get
\[ \AD{A^{(11)}} = \{ (d,\alpha_2), (c^2d,\alpha_2), (d,\alpha_3), (c^2d,\alpha_3) \}. \]
By \eqref{eq:SD}, we see that $\AD{A^{(11)}}$ coincides with $\SD{A^{(11)}}$. The following is easy to see.
\begin{lemma} \label{prp:A_11-autom}
We have $\HopfAuto(A^{(11)}) = \langle \tau, \varphi_u \mid u\in\kk^\times \rangle$, where $\tau(c)=c^3,\tau(d)=d,\tau(x)=x$, $\varphi_u|_{\G(A^{(11)})}=\id$ and $\varphi_u(x)=ux$.
\end{lemma}
Hence, we get $\#(\SD{A^{(11)}}) = \#(\SD{A^{(11)}}/\!\sim)$. Let $\Gamma$ be the group $C_4=\langle g\mid g^4=1\rangle$ of order four, and let $\chi$ be the character defined by $\chi(g)=\zeta_4$, where $\zeta_4\in\kk$ is a primitive fourth root of unity.
\begin{proposition}
The non-trivial super-forms of $A^{(11)}$ are the Hopf superalgebras $\AU(\Gamma,\DD)$ with
\[
\DD\in\{(1,\unit,0;1), (g^2,\unit,0;1), (1,\chi^2,0;1), (g^2,\chi^2,0;1)\}.
\]
They are pairwise non-isomorphic.
\end{proposition}
\begin{proof}
In the following, we determine the structure of each coinvariant subalgebra of $A^{(11)}$ using Theorem~\ref{thm:gen} one by one.
\begin{itemize}
\item The Hopf superalgebra $\HH_8^{(12)}:=(A^{(11)})^{\coo(d,\alpha_2)}$ is generated by $c,x$. As an element in $\HH_8^{(12)}$, $x$ is odd primitive. Thus, the assignment $g\mapsto c, z\mapsto x$ gives a Hopf superalgebra isomorphism $\AU(\Gamma,(1,\unit,0;1)) \cong \HH_8^{(12)}$. Note that $\HH_8^{(12)} \cong \kk C_4\otimes \bigwedge\kk$.
\item The Hopf superalgebra $\HH_8^{(14)}:=(A^{(11)})^{\coo(c^2d,\alpha_2)}$ is generated by $c,x$. As an element in $\HH_8^{(14)}$, $x$ is odd $c^2$-skew primitive. Thus, the assignment $g\mapsto c, z\mapsto x$ gives a Hopf superalgebra isomorphism $\AU(\Gamma,(g^2,\unit,0;1)) \cong \HH_8^{(14)}$.
\item The Hopf superalgebra $\HH_8^{(16)}:=(A^{(11)})^{\coo(d,\alpha_3)}$ is generated by $cd,x$. As an element in $\HH_8^{(16)}$, $x$ is odd primitive. Thus, the assignment $g\mapsto cd, z\mapsto x$ gives a Hopf superalgebra isomorphism $\AU(\Gamma,(1,\chi^2,0;1)) \cong \HH_8^{(16)}$.
\item The Hopf superalgebra $\HH_8^{(17)}:=(A^{(11)})^{\coo(c^2d,\alpha_3)}$ is generated by $cd,x$. As an element in $\HH_8^{(17)}$, $x$ is odd $c^2$-skew primitive. Thus, the assignment $g\mapsto cd, z\mapsto x$ gives a Hopf superalgebra isomorphism $\AU(\Gamma,(g^2,\chi^2,0;1)) \cong \HH_8^{(17)}$.
\end{itemize}
This completes the proof.
\end{proof}

\subsubsection{The $A^{(12)}$ case} \label{subsubsec:A^12}
Let $\zeta_4\in\kk$ be a primitive fourth root of unity. The Hopf algebra $A^{(12)}$ is given by
\[ A^{(12)}=\kk\langle c,d,x \mid c^4=d^2=1,cd=dc,x^2=0, cx=\zeta_4xc,dx=xd \rangle, \]
where $c$ and $d$ are group-like, and $x$ is $c^2d$-skew primitive.

First, note that both $\G(A^{(12)})$ and $\G((A^{(12)})^*)$ are isomorphic to $\Z_4\times \Z_2$ and are given by
\[ \G(A^{(12)}) = \langle c,d \rangle \quad\text{and}\quad \G((A^{(12)})^*) = \langle\alpha,\beta\rangle, \]
where $\alpha$ and $\beta$ are algebra maps $A^{(12)}\to\kk$ determined by $\alpha(c)=\zeta_4,\alpha(d)=1,\alpha(x)=0$ and $\beta(c)=1,\beta(d)=-1,\beta(x)=0$, respectively. Set $\alpha_1:=\alpha^2,\alpha_2:=\beta,\alpha_3:=\alpha^2\beta$. We get
\[ \AD{A^{(12)}} = \{ (d,\alpha_2), (c^2d,\alpha_2), (d,\alpha_3), (c^2d,\alpha_3) \}. \]
By \eqref{eq:SD}, we have $\SD{A^{(12)}} = \{ (c^2d,\alpha_2), (c^2d,\alpha_3) \}$. The following is easy to see.
\begin{lemma} \label{prp:A_12-autom}
We have $\HopfAuto(A^{(12)}) = \langle \tau, \varphi_u \mid u\in\kk^\times \rangle$, where $\tau(c)=cd,\tau(d)=d,\tau(x)=x$, $\varphi_u|_{\G(A^{(12)})}=\id$ and $\varphi_u(x)=ux$.
\end{lemma}
Hence, we get $\SD{A^{(12)}}/\!\sim = \{ [(c^2d,\alpha_2)] \}$. Let $\Gamma$ be the group $C_4=\langle g \mid g^4=1 \rangle$ of order four, and let $\chi$ be the character defined by $\chi(g)=\zeta_4$.
\begin{proposition}
The Hopf superalgebra $\AU(\Gamma,(1,\chi,0;1))$ is the only one with nonzero odd part whose bosonization is isomorphic to $A^{(12)}$.
\end{proposition}
\begin{proof}
The Hopf superalgebra $\HH_8^{(15)}:=(A^{(12)})^{\coo(c^2d,\alpha_2)}$ is generated by $c,x$. As an element in $\HH_8^{(15)}$, $x$ is odd primitive. Thus, the assignment $g\mapsto c, z\mapsto x$ gives a Hopf superalgebra isomorphism $\AU(\Gamma,(1,\chi,0;1)) \cong \HH_8^{(15)}$.
\end{proof}

\subsubsection{The $A^{(13)}$ case} \label{subsubsec:A^13}
The Hopf algebra $A^{(13)}$ is given by
\[ A^{(13)}=\kk\langle c,d,x \mid c^4=d^2=1,cd=dc,x^2=c^2-1, cx=-xc,dx=xd \rangle, \]
where $c$ and $d$ are group-like, and $x$ is $c$-skew primitive.

First, note that $\G(A^{(13)})\cong \Z_4\times\Z_2$ and $\G((A^{(13)})^*)\cong\Z_2\times\Z_2$. These groups are given by
\[ \G(A^{(13)}) = \langle c,d \rangle \quad\text{and}\quad \G((A^{(13)})^*) = \{\varepsilon,\alpha_1,\alpha_2,\alpha_3:=\alpha_1\alpha_2\}, \]
where $\alpha_1$ and $\alpha_2$ are algebra maps $A^{(13)}\to\kk$ determined by $\alpha_1(c)=-1,\alpha_1(d)=1,\alpha_1(x)=0$ and $\alpha_2(c)=1,\alpha_2(d)=-1,\alpha_2(x)=0$, respectively. Thus, we get
\[ \AD{A^{(13)}} = \{ (d,\alpha_2), (c^2d,\alpha_2), (d,\alpha_3), (c^2d,\alpha_3) \}. \]
By \eqref{eq:SD}, we have $\SD{A^{(13)}}=\varnothing$. Thus, we have the following result.
\begin{proposition}
The Hopf algebra $A^{(13)}$ has no non-trivial super-form.
\end{proposition}

\subsubsection{The $A^{(14)}$ case} \label{subsubsec:A^14}
The Hopf algebra $A^{(14)}$ is given by
\[ A^{(14)}=\kk\langle c,d,x \mid c^4=d^2=1,cd=dc,x^2=c^2-1, cx=xc,dx=-xd \rangle,\]
where $c$ and $d$ are group-like, and $x$ is $cd$-skew primitive.

First, note that $\G(A^{(14)})\cong \Z_4\times\Z_2$ and $\G((A^{(14)})^*)\cong\Z_2\times\Z_2$. These groups are given by
\[ \G(A^{(14)}) = \langle c,d \rangle \quad\text{and}\quad \G((A^{(14)})^*) = \{\varepsilon,\alpha_1,\alpha_2,\alpha_3:=\alpha_1\alpha_2\}, \]
where $\alpha_1$ and $\alpha_2$ are algebra maps $A^{(14)}\to\kk$ determined by $\alpha_1(c)=-1,\alpha_1(d)=1,\alpha_1(x)=0$ and $\alpha_2(c)=1,\alpha_2(d)=-1,\alpha_2(x)=0$, respectively. Thus, we get
\[ \AD{A^{(14)}} = \{ (d,\alpha_2), (c^2d,\alpha_2), (d,\alpha_3), (c^2d,\alpha_3) \}. \]
By \eqref{eq:SD}, we have $\SD{A^{(14)}} = \{ (d,\alpha_2), (c^2d,\alpha_2) \}$. The following is easy to see.
\begin{lemma} \label{prp:A_14-autom}
We have $\HopfAuto(A^{(14)}) = \langle \tau, \varphi_u \mid u\in\{1,-1\} \rangle$, where $\tau(c)=c^3,\tau(d)=c^2d,\tau(x)=x$, $\varphi_u|_{\G(A^{(14)})}=\id$ and $\varphi_u(x)=ux$.
\end{lemma}
Hence, we get $\SD{A^{(14)}}/\!\sim =\{ [(d,\alpha_2)] \}$. Let $\Gamma$ be the group $C_4=\langle g\mid g^4=1 \rangle$ of order four.
\begin{proposition}
The Hopf superalgebra $\AU(\Gamma,(g,\unit,1;1))$ is the only one with nonzero odd part whose bosonization is isomorphic to $A^{(14)}$.
\end{proposition}
\begin{proof}
The Hopf superalgebra $\HH_8^{(18)}:=(A^{(14)})^{\coo(d,\alpha_2)}$ is generated by $c,x$. As an element in $\HH_8^{(18)}$, $x$ is odd $c$-skew primitive. Thus, the assignment $g\mapsto c, z\mapsto\zeta_4x$ gives a Hopf superalgebra isomorphism $\AU(\Gamma,(g,\unit,1;1)) \cong \HH_8^{(18)}$.
\end{proof}

\begin{proof}[Proof of Theorem~\ref{prp:8-dim}]
Let $\HH$ be a non-semisimple pointed Hopf superalgebra of dimension $8$ with $\HH_{\bar1}\ne0$.
Its bosonization is a non-semisimple pointed Hopf algebra of dimension $16$.
By \cite{CaeDasRai00} and Proposition~\ref{prp:super-forms}, it is isomorphic to one of the Hopf algebras in Table~\ref{tab:CDR}.
The calculations in Sections~\ref{subsubsec:A^1}--\ref{subsubsec:A^14} determine all super-data up to equivalence for each of these Hopf algebras and give the nineteen examples in Table~\ref{tab:8-dim}.
For a fixed bosonization, distinct equivalence classes of super-data give non-isomorphic super-forms by the correspondence in Section~\ref{subsec:SD-SF}.
Super-forms with non-isomorphic bosonizations are also non-isomorphic.
Thus the nineteen examples are pairwise non-isomorphic and exhaust the possibilities.
\end{proof}

\section{Duals of pointed Hopf superalgebras} \label{sec:dual}
In this section, we also work over an algebraically closed field $\kk$ of characteristic zero. We have classified pointed Hopf superalgebras of dimension up to $10$. The aim of this section is to determine the duals of these Hopf superalgebras.

\subsection{Duals of $\HH_{2p}^{(i)}$} \label{subsec:dual-H_{2p}}
Let $p$ be an odd prime number, and let $\zeta_p\in\kk$ be a fixed primitive $p$-th root of unity. We identify the duals of the pointed Hopf superalgebras $\HH_{2p}^{(i)}$ ($i\in\{1,2,3,4\}$) introduced in Section~\ref{subsec:AN}.
\begin{theorem} \label{thm:H_2p^*}
There are isomorphisms
\[ (\HH_{2p}^{(1)})^* \cong \HH_{2p}^{(1)} \quad \text{and} \quad (\HH_{2p}^{(2)})^* \cong \HH_{2p}^{(3)} \]
of Hopf superalgebras. Moreover, the dual of $\HH_{2p}^{(4)}$ is non-pointed.
\end{theorem}
\begin{proof}
By our construction, there are Hopf algebra isomorphisms
\begin{equation*}
\widehat{\HH_{2p}^{(1)}} \cong \AA(-1,p,0), \quad \widehat{\HH_{2p}^{(2)}} \cong \AA(-1,1,0), \quad \widehat{\HH_{2p}^{(3)}} \cong \AA(-\zeta_p,p,0).
\end{equation*}
By Proposition~\ref{prp:dual-hat}, $\AA(-1,p,0)^*$ is isomorphic to the bosonization of $(\HH_{2p}^{(1)})^*$. Since $\AA(-1,p,0)$ is isomorphic to $T_{4}(-1) \otimes \kk \Z_p$, we see that $\AA(-1,p,0)^*$ is isomorphic to $\AA(-1,p,0)$. Since $\HH_{2p}^{(1)}$ is the unique non-trivial super-form of $\AA(-1, p, 0)$ up to isomorphism, we conclude that $(\HH_{2p}^{(1)})^*$ is isomorphic to $\HH_{2p}^{(1)}$.

The second isomorphism in the statement of this theorem is obtained in a similar manner: By Remark~(b) mentioned after \cite[Lemma~A.1]{AndNat01}, we see that $\AA(-1,1,0)$ and $\AA(-\zeta_p,p,0)$ are dual to each other. Since $\HH_{2p}^{(2)}$ and $\HH_{2p}^{(3)}$ are the unique non-trivial super-forms of $\AA(-1,1,0)$ and $\AA(-\zeta_p, p, 0)$, respectively, we conclude that $\HH_{2p}^{(2)}$ and $\HH_{2p}^{(3)}$ are dual to each other.

Finally, we mention the dual of $\HH_{2p}^{(4)}$. By Remark~(c) mentioned after \cite[Lemma~A.1]{AndNat01}, the dual of $\AA(-1,1,1)$ is not pointed. Hence, by Proposition~\ref{prp:pt/ss/chev}, $(\HH_{2p}^{(4)})^*$ is not pointed.
\end{proof}

\subsection{Duals of $\HH_4^{(i)}$} \label{subsec:dual-H_4}
We identify the duals of the pointed Hopf superalgebras $\HH_{4}^{(i)}$ ($i\in\{1,2,3,4\}$) introduced in Section~\ref{subsec:four-dim}.
\begin{theorem} \label{prp:dual-H_4}
There are isomorphisms of Hopf superalgebras
\[ (\HH_{4}^{(1)})^* \cong \HH_{4}^{(1)}, \quad (\HH_{4}^{(2)})^* \cong \HH_{4}^{(2)} \quad \text{and} \quad (\HH_{4}^{(3)})^* \cong \HH_{4}^{(4)}. \]
\end{theorem}
\begin{proof}
The first isomorphism follows from Example~\ref{ex:ext}. For $(i,j) = (2,2), (3,4)$, one can check that there is a non-degenerate Hopf pairing $\HH_4^{(i)} \times \HH_4^{(j)} \to \kk$ given by $\langle g,g \rangle=-1$, $\langle z,z \rangle = 1$, and $\langle g,z \rangle = \langle z, g\rangle =0$.
\end{proof}

\subsection{Duals of $\HH_8^{(i)}$} \label{subsec:dual-H_8}
We identify the duals of the pointed Hopf superalgebras $\HH_{8}^{(i)}$ ($i\in\{1,2,\dots,19\}$) introduced in Section~\ref{subsec:eight-dim}.
\begin{theorem}
There is an isomorphism $(\HH_8^{(i)})^* \cong \HH_8^{(j)}$ of Hopf superalgebras for each pair $(i,j) = (1,1)$, $(2,2)$, $(3,4)$, $(5,6)$, $(7,7)$, $(8,8)$, $(9,10)$, $(11,11)$, $(12,12)$, $(13,15)$, $(14,16)$, $(17,17)$. Moreover, the duals of $\HH_8^{(18)}$ and $\HH_8^{(19)}$ are non-pointed.
\end{theorem}
\begin{proof}
The case $(i,j)=(1,1)$ follows from Example~\ref{ex:ext}.
For $(i,j) = (3,4)$, $(7,7)$, $(13,15)$, the proof proceeds along the same lines as that of Theorem~\ref{thm:H_2p^*}. More precisely, for these pairs $(i,j)$, the bosonization of $\HH_8^{(i)}$ and that of $\HH_8^{(j)}$ are dual to each other by Proposition~\ref{prp:dual-hat} and \cite{Bea03}. According to our classification result, $\HH_8^{(\ell)}$ ($\ell\in\{i,j\}$) is the unique non-trivial super-form of its bosonization. Thus, by Proposition~\ref{prp:dual-hat}, we conclude that $\HH_8^{(i)}$ and $\HH_8^{(j)}$ are dual to each other.

For each of the following pairs $(i,j)$, a direct computation shows that the values below define a non-degenerate Hopf pairing $\langle \;,\; \rangle :\HH_8^{(i)} \times \HH_8^{(j)} \to \kk$.
\begin{itemize}
\item For $(i,j) = (2,2)$, $(5,6)$, define
\begin{equation*}
\langle g,g \rangle=-1, \quad
\langle z_s,z_t \rangle = \delta_{s,t} \quad\text{and}\quad
\langle g,z_s \rangle = \langle z_s,g\rangle =0 \quad (s, t \in\{1,2\}).
\end{equation*}
\item For $(i,j) = (8,8)$, $(9,10)$, define
\begin{equation*}
\langle g_s,g_t \rangle= (-1)^{\delta_{s,t}}, \quad
\langle z,z \rangle = 1 \quad\text{and}\quad
\langle g_s,z \rangle = \langle z,g_s\rangle = 0 \quad (s,t\in\{1,2\}).
\end{equation*}
\item For $(i,j)=(11,11)$, define
\[
\begin{gathered}
\langle g_1,g_1\rangle=\langle g_2,g_2\rangle=1,\quad
\langle g_1,g_2\rangle=\langle g_2,g_1\rangle=-1,\\
\langle z,z\rangle=1,\quad
\langle g_s,z\rangle=\langle z,g_s\rangle=0\quad(s\in\{1,2\}).
\end{gathered}
\]
\item For $(i,j) = (12,12)$, $(14,16)$, $(17,17)$, define
\begin{equation*}
\langle g,g \rangle=\zeta_4, \quad
\langle z,z \rangle = 1 \quad\text{and}\quad
\langle g,z \rangle = \langle z,g\rangle = 0,
\end{equation*}
where $\zeta_4$ is a primitive fourth root of unity.
\end{itemize}
The Hopf superalgebra $\HH_{8}^{(18)}$ is the only one with nonzero odd part whose bosonization is isomorphic to $A^{(14)}$. By \cite[Section~4.2, Table~2]{Bea03}, as corrected in \cite{Bea08}, the dual of $A^{(14)}$ is not pointed. Thus $(\HH_8^{(18)})^*$ is non-pointed.

For $\HH:=\HH_8^{(19)}$, we argue directly.
Put $e_\pm=(1\pm g)/2$.
As superalgebras,
\[
 e_+\HH\cong\bigwedge(\kk^2),
 \qquad
 e_-\HH\cong M(1|1),
\]
where $M(1|1)$ is $M_2(\kk)$ with the diagonal matrices even and the off-diagonal matrices odd.
Indeed, in the second summand one may take $x=2E_{12}$ and $y=E_{21}$.
The Jacobson radical of the underlying algebra is the homogeneous ideal
\[
 J=e_+\operatorname{span}_{\kk}\{x,y,xy\},
 \qquad
 \HH/J\cong\kk\oplus M(1|1).
\]
A simple $M(1|1)$-supermodule has dimension $1|1$, so $\HH^*$ is not pointed.
Moreover, let $q:\HH\to e_-\HH\cong M(1|1)$ be the projection.
Then $q(J)=0$, whereas
\[
 (q\otimes q)\Delta(e_+x)
 =I\otimes q(x)+q(x)\otimes I\ne0.
\]
Thus $J$ is not a coideal, and $\HH^*$ does not have the Chevalley property.
\end{proof}

\subsection{Concluding remarks}
The classification lists are given in Tables~\ref{tab:2p-dim}, \ref{tab:4-dim} and~\ref{tab:8-dim}.
Some of these Hopf superalgebras have non-pointed duals.
The classification of ordinary Hopf algebras of low dimension gives the following result.
\begin{theorem} \label{thm:non-ss,non-pt}
Suppose that $\HH$ is a non-semisimple non-pointed Hopf superalgebra such that $\HH_{\bar{1}} \ne 0$.
\begin{enumerate}
\item If $\dim(\HH) = 6$, then $\HH$ is isomorphic to $(\HH_{6}^{(4)})^*$.
\item If $\dim(\HH) = 8$ and $\HH$ does not have the Chevalley property, then $\HH$ is isomorphic to either $(\HH_8^{(18)})^*$ or $(\HH_8^{(19)})^*$.
\item If $\dim(\HH) = 10$, then $\HH$ is isomorphic to $(\HH_{10}^{(4)})^*$.
\end{enumerate}
\end{theorem}
\begin{proof}
We consider the case where $\dim(\HH)$ is either $6$ or $10$. According to Cheng and Ng~\cite[Theorem~II]{CheNg11}, a non-semisimple Hopf algebra of dimension $4p$ with $p\in\{3,5,7,11\}$ is pointed or dual-pointed (meaning that the dual is pointed). By Proposition~\ref{prp:pt/ss/chev} and the assumption that $\HH$ is non-semisimple and non-pointed, $\HH^*$ must be pointed. Thus $\HH^*$ is isomorphic to the Hopf superalgebra $\HH_{2p}^{(4)}$, where $p=\dim(\HH)/2$. Thus $\HH \cong (\HH_{2p}^{(4)})^*$.

Now suppose that $\dim(\HH)=8$ and $\HH$ does not have the Chevalley property.
By Proposition~\ref{prp:pt/ss/chev}, the bosonization $\widehat{\HH}$ does not have the Chevalley property.
By \cite[Theorem~1.2]{GarVay10}, its dual is pointed.
Propositions~\ref{prp:pt/ss/chev} and~\ref{prp:dual-hat} therefore imply that $\HH^*$ is pointed.
Theorem~\ref{prp:8-dim} and the duality statements above show that $\HH^*$ is isomorphic to either $\HH_8^{(18)}$ or $\HH_8^{(19)}$.
This proves (2).
\end{proof}

According to the classification of Hopf algebras of low dimension (see the survey \cite{BeaGar13}), non-semisimple Hopf algebras of dimension $\le 20$, except in dimension $16$, are pointed or dual-pointed. This implies that non-semisimple Hopf superalgebras of dimension $\le 10$, except in dimension $8$, are pointed or dual-pointed.

Given a primitive fourth root $\zeta_4$ of unity, we denote by $H_{16}(\zeta_4)$ the Hopf algebra of dimension $16$ with the Chevalley property introduced in \cite{CalDasMasMen04}. By \cite{GarVay10}, every non-semisimple Hopf algebra of dimension $16$ that is neither pointed nor dual-pointed has the Chevalley property and is isomorphic to either $H_{16}(\zeta_4)$ or $H_{16}(-\zeta_4)$. In \cite{ShiWak23}, we have determined all super-forms of $H_{16}(\pm\zeta_4)$ and obtained $8$ new Hopf superalgebras of dimension $8$. Thus, combining \cite{ShiWak23} with Theorems~\ref{thm:main1} and \ref{thm:non-ss,non-pt}, we have completed the classification of non-semisimple Hopf superalgebras of dimension $\le 10$. The remaining case, semisimple Hopf superalgebras of dimension $\le 10$, will be discussed in our forthcoming paper.

\providecommand{\bysame}{\leavevmode\hbox to3em{\hrulefill}\thinspace}
\providecommand{\MR}{\relax\ifhmode\unskip\space\fi MR }
\providecommand{\MRhref}[2]{%
  \href{http://www.ams.org/mathscinet-getitem?mr=#1}{#2}
}
\providecommand{\href}[2]{#2}

\end{document}